\documentclass[11pt,a4paper]{amsart}
\usepackage{fullpage}
\usepackage{epsfig, amsmath,xspace}
\usepackage{amssymb,amscd,mathrsfs}
\usepackage[all,cmtip]{xy}
\usepackage{hyperref}

\newtheorem{thm}{Theorem}[section]
\newtheorem{lem}[thm]{Lemma}
\newtheorem{prop}[thm]{Proposition}
\newtheorem{cor}[thm]{Corollary}
\theoremstyle{definition}
\newtheorem{rem}[thm]{Remark}
\newtheorem{defn}[thm]{Definition}
\newtheorem{ex}[thm]{Example}

\def\ie{\emph{i.e.}}
\def\eg{\emph{e.g. }}
\newdir{ >}{{}*!/-7pt/@{>}}

\def\R{\mathbb{R}}

\def\Der{\text{Der}}
\def\HH{\mathit{HH}}
\def\AQ{\mathit{AQ}}
\def\Z{\mathbb{Z}}
\def\Q{\mathbb{Q}}
\def\F{\mathbb{F}}
\def\grf2{\mathrm{gr}\mathbb{F}_2}
\def\grvctq{\mathrm{gr}\mathit{Vct}_\Q}

\def\leq{\leqslant}
\def\geq{\geqslant}
\def\lie{\mathrm{Lie}}
\def\ra{\rightarrow}
\def\rlie1{1rL}
\def\nlie{nL}
\def\lie{L}
\def\rg1{1rG}
\def\ng{nG}

\def\aug1rg{1rG_\varepsilon}

\begin{document}
\title{A spectral sequence for the homology of a finite algebraic delooping}
\author{Birgit Richter}
\author{Stephanie Ziegenhagen}
\address{Fachbereich Mathematik der Universit\"at Hamburg,
Bundesstra{\ss}e 55, 20146 Hamburg, Germany}
\email{birgit.richter@uni-hamburg.de}
\email{stephanie.ziegenhagen@uni-hamburg.de}
\keywords{$E_n$-homology, Hochschild cohomology, Andr\'e-Quillen
  homology, $\Pi$-Algebras, Grothendieck spectral sequence, Hodge
  decomposition}
\subjclass[2000]{Primary 16E40, 18G40; Secondary 55P35}

\begin{abstract}
In the world of chain complexes $E_n$-algebras are the analogues of
based $n$-fold loop spaces in the category of topological spaces.
Fresse showed that operadic $E_n$-homology of an $E_n$-algebra
computes the homology of an $n$-fold algebraic delooping. The aim of
this paper is to construct two spectral sequences for calculating
these homology groups and to treat some concrete classes of examples
such as Hochschild cochains, graded polynomial algebras and chains on
iterated loop spaces. In characteristic zero we gain an identification
of the summands in Pirashvili's Hodge decomposition of higher order
Hochschild homology in terms of derived functors of indecomposables of
Gerstenhaber algebras and as the homology of exterior and symmetric
powers of derived K\"ahler differentials.
\end{abstract}

\maketitle

\section{Introduction}

The little $n$-cubes operad acts on and detects based $n$-fold loop
spaces \cite{maygils}. Its algebraic counterpart, the operad that is
given by (a
cofibrant replacement of) its reduced chains, is the so-called
$E_n$-operad and its algebras are $E_n$-algebras. In this sense,
$E_n$-algebras are algebraic analogues of based $n$-fold loop
spaces. Benoit Fresse constructed an $n$-fold bar construction for any
$E_n$-algebra $A_*$, $B^nA_*$ \cite{F3} and showed that the
$E_n$-homology of $A_*$, $H^{E_n}_*(A_*)$ is the homology of the
$n$-fold desuspension of  $B^nA_*$, thus $E_n$-homology calculates the
homology of an algebraic $n$-fold delooping.

In characteristic zero the operad $E_{n+1}$ is quasi-isomorphic to  its
homology, $H_*(E_{n+1})$. The homology of $E_{n+1}$ codifies $n$-Gerstenhaber
algebras. As one consequence the operations on the homology of any
$E_{n+1}$-algebra are given by the $n$-Gerstenhaber algebra
structure.
However, in finite characteristic this does not hold any longer: If
$A_*$ is an algebra over the operad $E_{n+1}$, then $H_*A_*$ has more
structure than an $n$-Gerstenhaber algebra. The homology
of a free $E_{n+1}$-algebra on a chain complex $C_*$ carries a restricted
$n$-Gerstenhaber structure and in addition there are Dyer-Lashof operations to
consider. Thus in these cases it is not enough to study the
operad $H_*(E_{n+1})$ in order to understand all homology operations, but we have to understand the associated monad that maps a graded $k$-module $M_*$ to
$H_*(E_{n+1}(M_*))$.

We start in Part 1 by developing a standard resolution spectral sequence in the
cases of fields of characteristic two and zero. We identify its
$E^2$-term as the derived functor of indecomposables with respect to a
shifted (restricted) Gerstenhaber algebra structure (Theorem
\ref{thm:e2term} and Theorem \ref{thm:e2rational}). The chain complex
of an $n$-fold loop space carries an $E_n$-algebra
structure. If the loop space is of the form $\Omega^n \Sigma^n X$ for
$n \geq 2$ and connected $X$, then our spectral sequence gives an easy
argument for the fact that $E_n$-homology of the chain algebra
hands back the reduced homology of $\Sigma^n X$.

As we can express
the indecomposables with respect to a Gerstenhaber structure as a
composite of two functors we get a
Grothendieck-type spectral sequence in the non-additive context by the
work of Blanc and Stover \cite{BS}. This spectral sequence converges
to the input of the $E^2$-page of the resolution spectral sequence.

In Part 2 we apply these spectral sequences for calculations of
$E_n$-homology.

By forgetting structure every commutative algebra can be viewed as an
$E_n$-algebra.  In some cases classical work of Cartan \cite{C} can
be used to identify $E_n$-homology groups of commutative algebras. We
extend these classes of examples by calculating $E_n$-homology for free graded
commutative algebras on one generator.

A different class of interesting examples of $E_2$-algebras is the
class of reduced Hochschild cochain algebras of  associative algebras:
For any vector space $V$, the tensor
algebra $TV$ is the free associative algebra generated by $V$. Taking
the composition with the reduced  Hochschild cochains,
$\bar{C}^*(-,-)$, we assign to any vector
space $V$ the  $E_2$-algebra $\bar{C}^*(TV,TV)$. One can ask, how free this
$E_2$-algebra is. For a free $E_2$-algebra on a vector space $V$,
$E_2$-homology gives $V$ back. Is the homology of the $2$-fold
delooping, \ie, $H^{E_2}_*(\bar{C}^*(TV,TV))$, close to $V$? We give a
positive answer for a one-dimensional vector space over the rationals
(Theorem \ref{thm:hhfree-onedim})
and describe some partial results for the case when $V$ is of
dimension two.

In characteristic
zero the resolution spectral sequence for calculating $E_n$-homology
of commutative algebras has trivial differentials from
the $E^2$-term onwards and we get a decomposition of
$E_n$-homology. We identify
the summands of the Hodge
decomposition of higher order Hochschild homology in the sense of
\cite{P} with derived functors of indecomposables of Gerstenhaber
algebras. This recovers the identification of the Hodge summands of
Hochschild homology of odd order as
$$ \HH_{m+1}^{(\ell)}(A;\Q) \cong
H_{m-\ell+1}(\Lambda^{\ell} (\Omega^1_{P_*|\Q} \otimes_{P_*} \Q))$$
for a free simplicial resolution $P_*$ of the commutative algebra $A$
but we extend this result (see Theorem \ref{thm:hodge-aq}) to the Hodge
summands of Hochschild homology of even order:
$$
\mathrm{Tor}_{m+1-\ell}^\Gamma(\theta^\ell, \mathcal{L}(A;\Q)) \cong
(\mathbb{L}_mQ_{(2k-1)}\bar{A})_{(2k-1)(\ell-1)} \cong
H_{m-\ell+1}(\mathrm{Sym}^\ell (\Omega^1_{P_*|\Q} \otimes_{P_*} \Q)).$$

\textbf{Outline of the paper}: In section \ref{sec:resspsq} we construct a
resolution spectral
sequence and relate its $E^1$-term to Cohen's expression for the homology of
free $C_{n+1}$-spaces. We explain the corresponding algebraic notions in
section \ref{sec:algebra-f_2} (case of $\F_2$) and section \ref{sec:algebra-q}
(rational case) and use these to
identify the $E^2$-term of the resolution spectral sequence. Section
\ref{sec:appl} contains
a reduction result for certain $n$-Gerstenhaber algebras that have an
underlying free graded commutative algebra structure and we apply this to the
example of the $E_{n+1}$-algebra of rational chains on a based
$(n+1)$-fold iterated  loop space. We set up a Blanc-Stover spectral
sequence that converges
to  the $E^2$-term of our resolution spectral sequence in section \ref{sec:bs}.

As a first class of examples we present a calculation of $E_n$-homology of
free graded commutative algebras in section \ref{sec:freegraded}.
We discuss some calculations of $E_2$-homology of reduced Hochschild cochains
for tensor algebras and group algebras in section
\ref{sec:hhcochains}. Section \ref{sec:hodge}
contains our results about Pirashvili's Hodge decomposition for higher order
Hochschild homology where we relate his description of the decomposition
summands in terms of functor homology to homology groups of Gerstenhaber
algebras and to the homology of symmetric (exterior) powers of derived K\"ahler
differentials.

\bigskip

\textbf{Notation}: In the following we work relative to a field $k$, which is
most of the
time specified to be $\Q$ or $\F_2$. We denote by $k x$ the $k$-vector
space with basis element $x$.

\subsection*{Acknowledgement}
The first author thanks the Isaac Newton Institute for
  Mathematical Sciences in Cambridge for its hospitality. The second
  author gratefully acknowledges support by the DFG. We thank Haynes Miller for
  several helpful emails and conversations and Ib Madsen for
  mentioning Anderson's spectral sequence.

\part{Spectral sequences for $E_n$-homology}

\section{A spectral sequence} \label{sec:resspsq}
We choose a $\Sigma$-cofibrant model of the operad $E_{n+1}$ and as $n$
will be fixed in the following, we call this model $E=(E(r))$. Usually
$E_{n+1}(0)$ is used to keep track of base-points, but we
use the  reduced version with $E(0)=0$.
In the following we work with
augmented $E$-algebras $\varepsilon\colon A_* \ra k$ and $E_{n+1}$-homology is an
invariant of non-unital $E$-algebras, so we will consider the
augmentation ideal $\bar{A}_* = \ker(\varepsilon)$. However we will frequently
switch to working with $A_*$ when considering invariants of unital
objects, \eg  Andr\'e-Quillen homology.

We use a free simplicial resolution to establish a standard spectral
sequence converging to the
$E_{n+1}$-homology of any $E$-algebra.
\begin{lem}
For any $E$-algebra $\bar{A}_*$ there is a spectral sequence
$$E^1_{s,t} \cong H_t(E^{\circ s}(\bar{A}_*)) \Rightarrow
H^{E_{n+1}}_{s+t}(\bar{A}_*).$$
\end{lem}

\begin{proof}
As $\bar{A}_*$ is an $E$-algebra, there is a simplicial resolution of
$\bar{A}_*$ with
augmentation $\epsilon\colon E(\bar{A}_*) \ra \bar{A}_*$,
$$ \xymatrix{ {\cdots \, E^{\circ 3}(\bar{A}_*)}  \ar@<2ex>[r] \ar[r]
  \ar@<-2ex>[r] &   \ar@<1ex>[l]  \ar@<-1ex>[l]
  {E^{\circ 2}(\bar{A}_*)} \ar@<1ex>[r]  \ar@<-1ex>[r]
& \ar[l]  {E(\bar{A}_*)} \ar[r]^{\epsilon}& {\bar{A}_*}
}$$
The spectral sequence associated to the filtration by simplicial degree has
$$E^1_{s,t} = H^{E_{n+1}}_t(E^{s+1}(\bar{A}_*)) \Rightarrow H^{E_{n+1}}_*(\bar{A}_*).$$
But $E_{n+1}$-homology of a free $E_{n+1}$-algebra $E(B_*)$ is
isomorphic to $H_*(B_*)$ \cite[13.1.3, 4.4.2]{F1} and therefore the above
$E^1$-term reduces to $H_t(E^{\circ s}(\bar{A}_*))$.
\end{proof}

For topological spaces Cohen identified the homology of $C_{n+1}X$ for
any space $X$. Here $C_{n+1}$
denotes the operad of little $(n+1)$-cubes. He showed \cite[III,
Theorem 3.1]{CLM} that with $\F_p$-coefficients one gets
$$ H_*(C_{n+1}X; \F_p) \cong W_n(H_*(X;\F_p)).$$
Here $W_n$ is a free construction that takes the free restricted Lie
algebra structure, the partial
Dyer-Lashof structure and the commutativity of  $H_*(C_{n+1}X; \F_p)$
into account. \\

A similar description holds for the monad of homology operations in our algebraic setting. In \cite{F2} Fresse describes the
homology of $E(C_*)$ for any chain complex $C_*$. Note that
$$ H_*(E(C_*)) = H_*(\bigoplus_{r \geq 1} E(r) \otimes_{k[\Sigma_r]}
C_*^{\otimes r}) \cong \bigoplus_{r \geq 1} H_*(E(r) \otimes_{k[\Sigma_r]}
C_*^{\otimes r}).$$
We can view the term $H_*(E(r) \otimes_{k[\Sigma_r]}
C_*^{\otimes r})$ as the homology of a bicomplex and  as we assumed
that $E$ is $\Sigma$-cofibrant, the associated
spectral sequence has as vertical homology
$$E(r) \otimes_{k[\Sigma_r]} H_*(C_*^{\otimes r}).$$
As we are working over a field, there is a quasi-isomorphism from
$H_*(C_*)$ to $C_*$. Therefore there are no higher differentials
and we obtain the following isomorphism, which is natural in $C_*$.
\begin{lem} \label{lem:monad}
$$ H_*(E(C_*)) \cong \bigoplus_{r \geq 1} H_*(E(r) \otimes_{k[\Sigma_r]}
(H_*C_*)^{\otimes r}).$$
\end{lem}

For any space $X$ we denote by $C_*(X,\F_p)$ the
normalized chain complex of the simplicial $\F_p$-vector space of
simplices in $X$. The following result is
well-known; it is for instance used in  \cite{F2}, but for the
reader's convenience we record it with a proof.
Let  $\bar{W}_n$ denote the nonunital variant of $W_n$ determined by
$$\bar{W}_n(\bar{H}_*(X; \F_p)) = \bar{H}_*(C_{n+1} X ; \F_p).$$
\begin{prop} \label{prop:spacestochains}
For $k=\F_p$ there is an isomorphism of monads on the category of nonnegatively
graded $k$-modules 
\[
\bar{W}_n(-) \cong H_*(E(-)).
\]
\end{prop}
\begin{proof}
For any non-negatively graded $k$-module $M_*$  there exists a
(non-unique) based space $X^M$
such that the reduced homology of $X^M$ is
isomorphic to $M_*$. Let $E_+$ denote the unreduced
version of $E$ with $E_+(0)=\mathbb{F}_p$. Now $H_*(E(M_*))\cong H_*(\bigoplus_{r\geq 1} E(r) \otimes_{\mathbb{F}_p
  \Sigma_r} \tilde{H}_*(X^M; \mathbb{F}_p)^{\otimes r})$ coincides with
$H_*(\bigoplus_{r\geq 0} E_+(r) \otimes_{\mathbb{F}_p \Sigma_r}
\tilde{H}_*(X^M; \mathbb{F}_p)^{\otimes r} / E_+(0) \otimes
\tilde{H}_*(X^M; \mathbb{F}_p)^{\otimes 0})$. Switching to unreduced
homology can be done by introducing the quotient by base point
identification, hence the above is isomorphic to
\[
H_*(\bigoplus_{r\geq 0} E_+(r) \otimes_{\mathbb{F}_p \Sigma_r} H_*(X^M;
\mathbb{F}_p)^{\otimes r} / \sim)
\]
where $\sim$ reduces occurrences of the class $[pt] \in H_0(X^M;
\mathbb{F}_p)$  of the base point $pt \in X^M$ by contracting
the elements in the operad by inserting the basis element of $E_+(0) =
\mathbb{F}_p$  and divides out by $E_+(0) \otimes H_*(X^M ;
\mathbb{F}_p)^{\otimes 0} \cong \F_p$.
As we are working over a  field we can again replace the homology of
$X^M$ by its chain complex by picking
representatives for cycles. Since $\Sigma_r$ acts freely on
$C_{n+1}(r)$, the normalized singular chains $C_*(C_{n+1}(r),\F_p)$ are free as an
$\mathbb{F}_p \Sigma_r$-module. As $E_+$ is quasi-isomorphic to
$C_*(C_{n+1},\F_p)$  as an operad, we can identify the term
above with
$$
H_*(\bigoplus_{r \geq 0} C_*(C_{n+1}(r),\F_p) \otimes_{\mathbb{F}_p
  \Sigma_r} C_*(X^M,\F_p)^{\otimes r} / \sim)
$$
via the K\"unneth spectral sequence.
The fact that the shuffle transformation is lax symmetric monoidal
yields that the latter is isomorphic to
$$
H_*(C_*(\bigsqcup_{r \geq 0}  C_{n+1}(r)\times _{\Sigma_r} (X^M)^r /
\sim,\F_p) / C_*(C(0) \times (X^M)^0,\F_p))
$$
where $\sim$ now denotes the usual reduction of base points, so we see
that the above coincides with $\bar{H}_*(C_{n+1} X^M;\F_p)$.

Choosing a wedge of spheres for
$X^M$ and $X^N$ and observing that every morphism $M_* \rightarrow
N_*$ can be modelled via a map $X^M \rightarrow X^N$ one sees that
this isomorphism is natural in $M_*$.

To show that this is indeed an isomorphism of monads we first note
that the monad multiplication of $H_*(E(-))$ is induced by
the composition in $E$, whereas the monad multiplication of
$\bar{W}_n$ stems from the fact that $W_n$ is left adjoint to the
forgetful functor from what Cohen calls the category of allowable
$AR_n \Lambda_n$-Hopf algebras to $\mathbb{F}_p$-modules that are
unstable modules over the Steenrod algebra.

Iterating the isomorphism above yields
$$H_*(E(E(M_*))) \cong H_*(E(\tilde{H}_*(C_{n+1}X^M/\sim;\F_p))) \cong
\tilde{H}_*(C_{n+1}(C_{n+1}X^M/\sim)/\sim;\F_p).$$
Under this identification the multiplication of $H_*(E(-))$
corresponds to the map induced by the composition of
 $C_{n+1}$ and hence to the monad multiplication of $\bar{W}_{n+1}$.
\end{proof}

\begin{rem}
In the following we also need to cover the case where we consider
chain complexes that are concentrated in non-positive degrees, thus we
have to modify the
proof of Proposition \ref{prop:spacestochains}. To a graded $k$-module $M_{-*}$ concentrated in non-positive degrees we associate $M^*$ with $M^n = M_{-n}$, which is concentrated in
non-negative degrees. We choose a space
$X^M$ as above with $H_*(X^M; \F_p) \cong M^*$. As above we get an
isomorphism of monads between 
\begin{equation} \label{eq:cohomologymonad}
M_{-*} \mapsto H_*(E(M^*)) \text{ and }
M_{-*} \mapsto \bar{W}_n(M^*).
\end{equation}
\end{rem}

\section{The case $n=1,p=2$} \label{sec:algebra-f_2}
Cohen showed in \cite[Theorem 3.1]{CLM} that over a prime field the
homology of a free $C_{n+1}$-algebra in spaces, $C_{n+1}X$, can be
described as a free
gadget on the homology of the underlying space. He also gave
a description that allows to deduce the answer in characteristic
zero.

For odd primes or for higher iterated loop spaces, the answer is
pretty involved, but for $p=2$ and $n=1$ one
is left with a $1$-restricted Gerstenhaber structure.

Haynes Miller worked out the case $p=2,n=1$ in \cite{Mi2}: For any
space, the homology of $C_2X$ with coefficients in $\mathbb{F}_2$ is
given by
$$S(1rL)(\tilde{H}_*(X;\mathbb{F}_2))$$
where $1rL(-)$ denotes the free $1$-restricted Lie algebra and $S(-)$ is the free graded commutative algebra with the induced unique restricted Lie structure. We will recall the definitions and fix some notation.

\begin{defn} \label{def:1rl}
A \emph{$1$-restricted Lie algebra over $\mathbb{F}_2$} is a
non-negatively graded (or non-positively graded) $\mathbb{F}_2$-vector
space, $\mathfrak{g}_*$,
together with two operations, a Lie bracket of degree one, $[-,-]$ and
a restriction, $\xi$:
$$\begin{array}{rlc}
[-,-]\colon & \mathfrak{g}_i \times \mathfrak{g}_j \ra
\mathfrak{g}_{i+j+1}, & i,j  \geq 0, \\
\xi \colon & \mathfrak{g}_i \ra \mathfrak{g}_{2i+1} & i \geq 0.
\end{array}$$
These satisfy the relations
\begin{enumerate}
\item
The bracket is bilinear, symmetric and satisfies the Jacobi relation
$$[a,[b,c]] + [b,[c,a]] + [c,[a,b]] = 0 \text{ for all
}
a,b,c \in \mathfrak{g}_*.$$
\item
The restriction interacts with the bracket as follows:
$[\xi(a),b] = [a,[a,b]]$ and $\xi(a+b) = \xi(a) + \xi(b) + [a,b]$ for
all
homogeneous $a,b \in \mathfrak{g}_*$.
\end{enumerate}
A \emph{morphism of $1$-restricted Lie-algebras} is a map of graded
vector spaces of degree zero preserving the bracket and the
restriction. We denote the category of  $1$-restricted Lie-algebras by
$\rlie1$.
\end{defn}
\begin{rem}\label{rem:1rl}
Note that these relations imply that $[a,a] = 0$ and $\xi(0)=0$.
\end{rem}
\begin{defn} \label{def:1rg}
A \emph{$1$-restricted Gerstenhaber algebra over $\mathbb{F}_2$} is a
$1$-restricted Lie algebra $G_*$ together with an augmentation
$\varepsilon\colon G_* \ra \F_2$ and a graded commutative
$\mathbb{F}_2$-algebra structure on $G_*$ such that the
multiplication in $G_*$ interacts with the restricted Lie-structure as
follows:
\begin{itemize}
\item[]
\item
(Poisson relation)
$$ [a,bc] =[a,b]c +  b[a,c]  \text{ for all
} a,b,c \in G_*.$$
\item
(multiplicativity of the restriction)
$$ \xi(ab) = a^2\xi(b) + \xi(a)b^2 + ab[a,b] \text{ for all
homogeneous } a,b \in G_*.$$
\end{itemize}
The augmentation is required to be multiplicative and to satisfy
$\varepsilon[a,b]=0$ for all $a,b \in G_*$ and $\varepsilon\xi(a)=0$.

A \emph{morphism of $1$-restricted Gerstenhaber algebras} is a map of
graded vector spaces of degree zero preserving the product, the
augmentation, the bracket and the restriction. We denote the category
of $1$-restricted Gerstenhaber algebras by $\rg1$.
\end{defn}
In particular, the bracket and the restriction annihilate squares:
$[a,b^2] = 2b[a,b] = 0$ and $\xi(a^2) = 2a^2\xi(a) + a^2[a,a] =
0$. Thus if $1$ denotes the unit of the algebra structure in $G_*$,
then $[a,1] = 0$ for all $a$ and $\xi(1)=0$.

Usually an augmentation
is not part of the definition, but since we consider augmented
$E$-algebras all $1$-restricted Gerstenhaber
algebras that we will encounter are naturally augmented. The
requirements on $\varepsilon$ are equivalent to $\varepsilon$ being a
morphism of $1$-restricted Gerstenhaber algebras $\varepsilon\colon
G_* \ra \F_2$ where $\F_2$ is viewed as a commutative algebra with
trivial $1$-restricted Lie structure.

We denote by
$IG_*$ the augmentation ideal of $G_*$. This ideal carries a structure
of a non-unital $1$-restricted Gerstenhaber algebra and we will
call both $G_*$ and $IG_*$ $1$-restricted Gerstenhaber algebras.

For a $1$-restricted Lie-algebra $\mathfrak{g}$, the free graded
commutative algebra generated by $\mathfrak{g}$, $S(\mathfrak{g})$,
carries a unique $1$-restricted Gerstenhaber  structure that is
induced by the $1$-restricted Lie algebra structure on $\mathfrak{g}$
and the relations in Definition \ref{def:1rg}.

\begin{rem}
The functor $S\colon \rlie1 \ra \rg1$ is left adjoint to the
augmentation ideal functor $I\colon \rg1 \ra \rlie1$ and the forgetful
functor $U\colon \rlie1 \ra \grf2$ from the category of $1$-restricted
Lie-algebras to the category of graded $\F_2$-vector spaces has the
free $1$-restricted Lie algebra functor $1rL\colon \grf2 \ra \rlie1$ as
a left adjoint.
\end{rem}

For $p=2,n=1$ the structure on $H_*(C_2X;\F_2)$ looks so nice because
$$R_1(q) = \mathbb{F}_2\langle Q^I| I=(s_1,\ldots,s_k) \text{
  admissible, } e(I) \geq q, s_k \leq q\rangle$$
reduces to
$$R_1(q) = \mathbb{F}_2\langle  Q^{(2^{k-1}q,\ldots,2q,q)}|k \geq
1 \rangle.$$
Therefore, in Cohen's identification of $H_*(C_2X;\mathbb{F}_2)$ the
contribution of the Dyer-Lashof terms is absorbed into the free
commutative algebra part: a term like
$$ Q^{(2^{k-1}q,\ldots,2q,q)} \otimes x_q$$
with $x_q$ in $1rL(H_*X;\mathbb{F}_2)$ of degree $q$ is identified (in
what Cohen
calls $V_1$) with $x_q^{2^k}$. Thus $H_*(C_2X;\mathbb{F}_2) \cong
S(1rL)(\tilde{H}_*(X;\F_2))$.

As a corollary to Lemma \ref{lem:monad} and Proposition \ref{prop:spacestochains} we get
\begin{cor}
For any non-negatively graded (or non-positively graded) chain complex
over $\F_2$, $C_*$, we have
$$ H_*(E_2C_*;\F_2) \cong IS(1rL)(H_*C).$$
\end{cor}

\begin{defn}
For an augmented $1$-restricted Gerstenhaber algebra over
$\mathbb{F}_2$ we denote by $Q_{1rG}$ the $\mathbb{F}_2$-vector space
of indecomposable elements with respect to the three operations, the
product, the bracket and the restriction, \ie, the quotient of  $IG_*$
by the ideal generated by these operations:
$$ Q_{1rG}(G_*) = IG_*/\langle\xi(a), [a,b], ab, a,b \in
IG_*\rangle.$$
We extend this notion to $IG_*$, so
$Q_{1rG}(IG_*) = IG_*/\langle\xi(a), [a,b], ab, a,b \in
IG_*\rangle$.
\end{defn}

Similarly, we denote by $Q_{a}(-)$ the indecomposables with respect to
the algebra structure and by $Q_{1rL}(-)$ the indecomposables with
respect to the $1$-restricted Lie algebra structure.
\begin{lem}
For any augmented $1$-restricted Gerstenhaber algebra $G_*$ over
$\F_2$ the vector space of indecomposables $Q_{1rG}$ of $G_*$ can be
computed as the composite
$$ Q_{1rG}(G_*) = Q_{1rL}(Q_{a}(G_*)).$$
\end{lem}
\begin{proof}
As we demand that $\varepsilon$ annihilates Lie brackets and
restrictions, there is a well-defined $1$-restricted
Lie-structure on $IG_*$ and the algebra indecomposables, $Q_{a}(G_*) =
IG_*/J$, inherit a $1$-restricted Lie algebra structure from $G_*$:
For homogeneous $a,b \in \bar{G}_*$ we set
$$ [a+J,b+J] := [a,b] +J \text{ and } \xi(a + J) := \xi(a)+J.$$
The relations from Definition \ref{def:1rg} tell us that this  gives a
well-defined bracket and a well-defined restriction on
$Q_a(G_*)$. Taking the $1$-restricted Lie algebra indecomposables of
$Q_a(G_*)$ kills expressions in $Q_a(G_*)$ that are of the form
$\xi(a)$ with $a \in IG_*$ and $[a,b]$ with $a,b \in G_*$, so we kill
everything in $IG_*$ that is a product, a bracket or  a restricted
element.
\end{proof}

The algebraic indecomposables of a free commutative algebra on a
(graded) vector space hand back the vector space and the
indecomposables with respect to the $1$-restricted Lie algebra
structure of $1rL(V_*)$ have $V_*$ as output, so we get:
\begin{lem}
$$ Q_{1rG}(S(1rL)(V_*)) = V_* \text{ and } Q_{1rG}(IS(1rL)(V_*)) = V_*. $$
\end{lem}

We want to identify the $E^2$-term of the spectral sequence
$$ E^1_{p,q}=H_q(E_2^{p}(\bar{A}_*))  \Rightarrow
H^{E_2}_{p+q}(\bar{A}_*). $$

Lemma $\ref{lem:monad}$ allows us to identify $H_q(E_2^{p}(\bar{A}_*)) $ with $H_q(E_2^{p}(H_*(\bar{A}_*)))$, while Proposition $\ref{prop:spacestochains}$ yields that this coincides with  $((IS(1rL))^p(H_*(\bar{A}_*)))_q$.
The homology of $\bar{A}_*$ is a non-unital $1$-restricted
Gerstenhaber algebra over $\F_2$, so the free-forgetful adjunction identifies
$$ \ldots \ra (IS(1rL))^{p+1}(H_*(\bar{A}_*)) \ra
(IS(1rL))^p(H_*(\bar{A}_*)) \ra \ldots \ra IS(1rL)(H_*(\bar{A}_*))$$
as a resolution coming from a simplicial resolution of
$H_*(\bar{A}_*)$. The term $(IS(1rL))^{p+1}(H_*(\bar{A}_*))$ is
in resolution degree $p$. Applying the functor $Q_{1rG}$ to this simplicial
resolution gives
$$ \ldots \ra (IS(1rL))^{p}(H_*(\bar{A}_*)) \ra
(IS(1rL))^{p-1}(H_*(\bar{A}_*)) \ra \ldots
\ra H_*(\bar{A}_*).$$
Comparing the faces and degeneracies of this simplicial resolution with the image of the differential $d^1$ under the identification $H_q(E_2^{p}(\bar{A}_*))  \cong
((IS(1rL))^p(H_*(\bar{A}_*)))_q$ shows:
\begin{thm} \label{thm:e2term}
The above $E^1$-term is isomorphic to
$$ E^1_{p,q} \cong (Q_{1rG}((IS(1rL))^{p+1}(H_*(\bar{A}_*)))_q$$
and the $d^1$-differential takes homology with respect to the
resolution degree. Therefore the $E^2$-term calculates derived
functors of indecomposables of the homology of $\bar{A}_*$,
$$ E^2_{p,q} \cong (\mathbb{L}_pQ_{1rG}(H_*(\bar{A}_*)))_q.$$
\end{thm}

\section{The rational case} \label{sec:algebra-q}
Most things are similar for the rational case  with the difference
that we consider $n$-Gerstenha\-ber algebras for all $n \geq 1$. In this
section the ground field will always be $\Q$.

\begin{defn} \label{def:nlie}
An \emph{$n$-Lie algebra over $\Q$} is a non-negatively graded (or
non-positively graded) $\Q$-vector space, $L_*$,  together with a Lie
bracket of degree $n$:
$$
[-,-]\colon  L_i \times L_j \ra L_{i+j+n}, \, i,j  \geq 0,
$$
such that the bracket is bilinear  and satisfies a graded
Jacobi relation
$$(-1)^{pr}[x,[y,z]] + (-1)^{qp}[y,[z,x]] + (-1)^{rq}[z,[x,y]] = 0,$$
and a graded antisymmetry relation
$$ [x,y] = -(-1)^{pq}[y,x].$$
Here,
$x,y,z$ are homogenous elements in $L$ and
 $p=|x|+n$, $q=|y|+n$ and $r=|z|+n$.
A \emph{morphism of $n$-Lie algebras} is a map of graded vector spaces
of degree zero preserving the bracket. We denote the category of
$n$-Lie algebras by $\nlie$.
\end{defn}
Note that there is an operadic notion of $n$-Lie algebras involving
$n$-ary Lie brackets. That is something different.

\begin{defn}\label{def:ng}
An \emph{$n$-Gerstenhaber algebra over $\Q$} is an $n$-Lie algebra
$G_*$ together with a unital commutative $\Q$-algebra structure on
$G_*$ and an augmentation $\varepsilon\colon G_* \ra \Q$
such that the Poisson relation holds:
$$ [a,bc] = [a,b]c +(-1)^{p|b|} b[a,c] , \text{ for all homogeneous
} a,b,c \in G_* \text{ with } p= |a|+n, $$
and such that $\varepsilon[a,b]=0$.

A \emph{morphism of $n$-Gerstenhaber algebras} is a map of graded
vector spaces of degree zero preserving the product, the augmentation and the
bracket. We denote the category of (augmented) $n$-Gerstenhaber algebras
by $\ng$. As in the characteristic two case, we also consider $IG_*$
as a non-unital $n$-Gerstenhaber algebra.

Let $nG$ denote the free $n$-Gerstenhaber algebra functor from the
category $\grvctq$ of graded rational vector spaces to the category of
augmented $n$-Gerstenhaber algebras. Then this can be factored as $S
\circ nL$ where $nL$ denotes  the free $n$-Lie algebra functor.

Similarly, we can factor the functor of $n$-Gerstenhaber
indecomposables, $Q_{nG}$, as the algebraic indecomposables followed
by the $n$-Lie indecomposables:
$$ Q_{nG} = Q_{nL} \circ Q_a.$$
\end{defn}
\begin{thm} \label{thm:e2rational}
There is a spectral sequence with
$$ E^2_{p,q} \cong (\mathbb{L}_pQ_{nG}(H_*(\bar{A}_*)))_q \Rightarrow
H^{E_{n+1}}_*(\bar{A}_*)$$
for every $E_{n+1}$-algebra $\bar{A}_*$ over the rationals.
\end{thm}

\begin{ex}
Let $X$ be a connected and well-behaved topological space and let $n$
be greater or equal to  one. In characteristic zero,
$H_*(C_{n+1}X;\Q)$ is isomorphic to the free $n$-Gerstenhaber algebra
generated by the reduced homology of $X$ \cite{CLM}. Thus the above
$E^2$-term for $A_* = C_*(C_{n+1}X;\Q)$ is isomorphic to
$$ E^2_{p,q} \cong (\mathbb{L}_pQ_{nG}(nG(\bar{H}_*(X;\Q)))_q$$
which is concentrated in bidegrees $(0,q)$ and thus the spectral
sequence collapses and gives
$$ H^{E_{n+1}}_q(C_*(C_{n+1}X;\Q)) \cong \bar{H}_q(X;\Q).$$

Note that $H_*(C_{n+1}X;\Q) \cong H_*(\Omega^{n+1}\Sigma^{n+1} X, \Q)$ if $X$ is
path-connected, thus in this case the algebraic delooping induced by
$E_{n+1}$-homology
corresponds to a geometric delooping.

Similar considerations hold for $H_*(C_{2}X;\F_2)$.
\end{ex}

\section{Some applications of the resolution spectral sequence}
\label{sec:appl}

Let  $U$ be the forgetful functor from the category of $n$-Lie
algebras to graded $\Q$-vector spaces.
\begin{lem} \label{lem:1gres}
Let $V$ be an $n$-Lie algebra and let $C$ be $S(V)$ where the
$n$-Gerstenhaber algebra structure on $C$ is induced by the $n$-Lie structure
of $V$. Then a free resolution in the
category of simplicial $n$-Gerstenhaber algebras of $C$ is given by
$Y_\bullet$ with
$$ Y_\ell = S (nL \circ U)^{\ell +1}(V), \, \ell \geq 0.$$
\end{lem}
\begin{proof}
We use the adjunction $(nL,U)$ to obtain the
simplicial structure on $Y_\bullet$. The usual simplicial contraction
for $(nL\circ U)^{\bullet+1}$ shows that $Y_\bullet$ is a resolution
of $S(V)$.
Note that the augmentation
$$ Y_0 = S(nL \circ U)(V)  \ra S(V)$$
is a morphism of $n$-Gerstenhaber algebras.

The degeneracy maps send $nG(U(nL\circ U)^\ell(V))$ to $nG(U(nL\circ
U)^{\ell+1}(V))$ via maps of the form $nG(f)$ with $f\colon U(nL\circ
U)^\ell(V) \ra U(nL\circ U)^{\ell+1}(V)$, thus $Y_\bullet$ is a free
simplicial resolution of $S(V)$.
\end{proof}

\begin{cor} \label{cor:e2smooth}
For $S(V)$ as in Lemma \ref{lem:1gres} there is an isomorphism
$$ (\mathbb{L}_pQ_{nG}(S(V)))_q \cong (\mathbb{L}_pQ_{nL}(V))_q$$
for all $p \geq 0$ and all $q$. In particular, for any
$E_{n+1}$-algebra $A_*$ with $H_*(A_*) \cong S(V)$ with $S(V)$
as in Lemma \ref{lem:1gres}, the
$E^2$-term of Theorem \ref{thm:e2rational}  is
isomorphic to
$$E^2_{p,q} \cong (\mathbb{L}_pQ_{nL}(V))_q.$$
\end{cor}
\begin{proof}
Using the resolution $S(nL\circ U)^{\bullet+1}(V)$ of $S(V)$ we get
that
$$ Q_{nG}(S(nL\circ U)^{s+1}(V)) \cong U((nL\circ U)^{s})(V) =
Q_{nL}((nL\circ U)^{s+1}(V)).$$
As $(nL\circ U)^{\bullet+1}(V)$ is a simplicial resolution of $V$ by free $n$-Lie
algebras, the claim follows.
\end{proof}

\begin{rem} \label{rem:suspensions}
There is an equivalence of categories between the category $\nlie$ and
the category of graded Lie algebras, $\lie$, where the latter category
is nothing but the category of $0$-Lie algebras \cite[Proposition
I.6.3]{KM}. The equivalence is given by the $n$-fold suspension,
$\Sigma^n$,  and desuspension, $\Sigma^{-n}$:
$$ \xymatrix@1{ {\nlie \, } \ar@<0.8ex>[rr]^{\Sigma^n} & &
  \ar@<0.8ex>[ll]^{\Sigma^{-n}} { \, \lie.}}$$
An analogous result holds in characteristic two \cite[III \S 15]{CLM}.
\end{rem}

We can use
the resolution exhibited in \ref{cor:e2smooth} to exploit
the equivalences between the different $n$-Lie structures.
\begin{cor} \label{cor:shift}
Suppose that $H_*(A)= S(V)$ as in \ref{lem:1gres}.
Then for
every
$\ell \in \mathbb{Z}$ the $E^2$-term of the spectral sequence
calculating $E_{n+1}$-homology of $A$ can be computed as
$$E^2_{p,q} = (\mathbb{L}_p Q_{\ell L} (\Sigma^{n-\ell}V))_{q+n-\ell}.$$
\end{cor}
\begin{proof}
Using the natural isomorphism $nL\cong \Sigma^{\ell-n} \ell L
\Sigma^{n-\ell}$ and the fact that this is an isomorphism of monads we
find
\begin{eqnarray*}
E^2_{p,q} &=& (\pi_p U(nL U)^{\bullet}(V))_q \cong (\pi_p
U(\Sigma^{\ell-n} \ell L\Sigma^{n-\ell}U)^{\bullet}(V))_q\\
&\cong & (\pi_p U(\Sigma^{\ell-n} \ell LU\Sigma^{n-\ell})^{\bullet}
(V))_q \cong (\pi_p  \Sigma^{\ell-n} U(\ell LU)^{\bullet}
(\Sigma^{n-\ell} V))_q\\
&\cong & (\pi_p  U(\ell LU)^{\bullet} (\Sigma^{n-\ell} V))_{q+n-\ell},
\end{eqnarray*}
which proves the claim.
\end{proof}

\begin{rem}\label{rem:e2smoothchartwo}
Similar results hold in characteristic two, \ie, if $V$ is a
$1$-restricted Lie algebra over $\F_2$ and if $C=S(V)$ carries the
induced $1$-restricted Gerstenhaber algebra structure, then $S(1rL
\circ U)^{\bullet +1}$ is a resolution of $C$ by free $1$-restricted
Gerstenhaber algebras and the $E^2$-term of the resolution spectral
sequence simplifies to
$$ E^2_{p,q} \cong (\mathbb{L}_pQ_{1rL}(V))_q$$
and the suspension isomorphism yields that this in turn can be
expressed as derived functors of indecomposables of restricted Lie
algebras.
\end{rem}

In the following we want to use the Tor interpretation of
Lie-homology in our setting:
\begin{rem}\label{rem:lieviator}
Let $\mathfrak{g}$ be a graded Lie-algebra, restricted if $k=\F_2$ and unrestricted over the rationals.
The usual Tor interpretation (see for instance \cite{Q2}) of
$\mathbb{L}_s Q_{L} (\mathfrak{g})$ holds in the graded case. Indeed one easily identifies the indecomposables of the
standard cofibrant replacement $X=L^{\bullet +1}(\mathfrak{g})$ (or  $X=rL^{\bullet +1}(\mathfrak{g})$ in the restricted case) with
$\overline{\mathfrak{U}(X)} \otimes_{\mathfrak{U}(X)} k$. Here
$\mathfrak{U}(\mathfrak{g})$ denotes the (restricted) universal
enveloping algebra of the bigraded Lie algebra $\mathfrak{g}$ while
$\overline{\mathfrak{U}(\mathfrak{g})}$ denotes its augmentation
ideal. The K\"unneth
spectral sequence constructed by Quillen \cite[II.6]{Q} can be
generalized to the graded setting, and since  $X$ consists of free
graded (restricted) Lie algebras, $\overline{\mathfrak{U}(X)}$ is a cofibrant
$\mathfrak{U}(X)$-module. Hence we get a spectral sequence of
internally bigraded vector spaces
$$
E^2_{p,q}=\mathrm{Tor}_p^{\pi_*(\mathfrak{U}(X))}
(\pi_*(\overline{\mathfrak{U}(X)}),k)_q \Rightarrow
\pi_{p+q}(\overline{\mathfrak{U}(X)}
\otimes_{\mathfrak{U}(X)} k),$$
where $q$ is the degree originating from taking homotopy groups.
Filtering $\mathfrak{U}(X)$ and  $\mathfrak{U}(\mathfrak{g})$ by the standard
filtration for enveloping
algebras and considering the associated spectral sequences, a bigraded
version of the Poincar\'e-Birkhoff-Witt theorem shows that the augmentation
$\mathfrak{U}(X)\rightarrow \mathfrak{U}(\mathfrak{g})$ induces an
isomorphism on $E^1$ (see \cite{Pr} for the case of characteristic
two). Hence the above $E^2$-term equals
$\mathrm{Tor}_p^{\mathfrak{U}(\mathfrak{g})}(\overline{\mathfrak{U}(\mathfrak{g})},k)$
concentrated in degree $q=0$. Finally the short exact sequence
$$\xymatrix{ 0 \ar[r] &\overline{\mathfrak{U}(\mathfrak{g})} \ar[r] &
  \mathfrak{U}(\mathfrak{g}) \ar[r] & k \ar[r] & 0}$$
yields
$$\mathbb{L}_s Q_{L} (\mathfrak{g})\cong
\mathrm{Tor}_{s+1}^{\mathfrak{U}(\mathfrak{g})}(k,k).$$
\end{rem}

An example of how this simplifies our  spectral sequence is given
by the chains on an iterated loop space on
a highly connected space. For an $(n+1)$-connected space $X$  the
space $\Omega^{n+1}X$ is
path-connected. A classical result expresses $H_*(\Omega^{n+1}X;\Q)$
as a free graded commutative algebra: The connectivity assumptions
ensure that due to the Milnor-Moore result \cite[p.263]{MM} the Hurewicz map
$$ \pi_*(\Omega^{n+1}X) \otimes \Q \ra H_*(\Omega^{n+1}X;\Q)$$
induces an isomorphism of Hopf algebras between the enveloping algebra
of the Lie-algebra $\pi_*(\Omega^{n+1}X) \otimes \Q$ and
$H_*(\Omega^{n+1}X;\Q)$. Here, the Lie-structure on the source is
given by the Samelson product. For $n \geq 1$, this Lie-structure is
trivial and thus the enveloping algebra is isomorphic to the free
graded commutative algebra generated by $\pi_*(\Omega^{n+1}X) \otimes
\Q$:
$$ S(\pi_*(\Omega^{n+1}X) \otimes \Q) \cong H_*(\Omega^{n+1}X;\Q).$$
Cohen showed \cite[p.~215]{CLM} that the Whitehead product on
$ \Sigma^{-n-1}\pi_*(X) \otimes \Q$ corresponds to the Browder-bracket
$\lambda_n$ on $H_*(\Omega^{n+1}X;\Q)$.
Gaudens and Menichi observed  \cite[Theorem 4.1]{GM} that this leads to
an isomorphism of $nG$-algebras
\begin{equation} \label{eq:gniso}
S(\Sigma^{-n}\pi_*(\Omega X) \otimes
\Q) \cong H_*(\Omega^{n+1}X;\Q)
\end{equation}
where the $n$-Lie structure on the left-hand side is induced by the
Samelson bracket on $\pi_*(\Omega X)$.
\begin{prop}
For every $(n+1)$-connected space $X$
$$ \mathbb{L}_s(Q_{nG})(H_*(\Omega^{n+1}X;\Q))_q \cong
\mathrm{Tor}_{s+1, q+n}^{H_*(\Omega X;\Q)}(\Q,\Q). $$
\end{prop}
\begin{proof}
Corollary \ref{cor:e2smooth} implies that the $E^2$-term of the resolution
spectral sequence in this case is isomorphic to
$$ (\mathbb{L}_s Q_{nL})_t(\Sigma^{-n}(\pi_*(\Omega X) \otimes \Q))).$$
Corollary \ref{cor:shift}  together with the
$\mathrm{Tor}$-description of Lie-homology and the
Milnor-Moore Theorem show the claim.
\end{proof}
\begin{rem}
Up to a shift in degrees the above $E^2$-term is isomorphic to the
$E^2$-term of the Rothenberg-Steenrod spectral sequence \cite{RS}. The latter
converges to the homology of the space $X$. We conjecture that there
is an isomorphism of spectral sequences between our resolution
spectral sequence and the (shifted) Rothenberg-Steenrod spectral sequence.
\end{rem}
\begin{rem}
Anderson constructed a spectral sequence \cite{And} whose $E^2$-page
is
$$ E^2_{p,q} \cong \HH^{[n+1]}_p(H_*(\Omega^{n+1}X;\Q))_q$$
and which converges to $H_{p+q}(X;\Q)$. Here $\HH^{[n+1]}_*$ denotes
Hochschild homology of order $n+1$ in the sense of \cite{P}.

However, in his setting $H_*(\Omega^{n+1}X;\Q)$ is considered as a
graded commutative algebra, whereas the $n$-Lie structure is ignored.
In this situation Hochschild homology of order $n+1$ is isomorphic to
$E_{n+1}$-homology,
$$\HH^{[n+1]}_p(H_*(\Omega^{n+1}X;\Q))_q \cong
H^{E_{n+1}}_{p-n-1}(\bar{H}_*(\Omega^{n+1}X;\Q))_q.$$
Thus his spectral sequence starts off with
$E_{n+1}$-homology of the underlying graded commutative algebra of
$H_*(\Omega^{n+1}X;\Q)$ and converges to $H_*(X;\Q)$.

\end{rem}
\section{The Blanc-Stover composite functor spectral sequence}
\label{sec:bs}

We know that working relative to $\F_2$ we can factor $Q_{1rG}$ as
$Q_{1rL} \circ Q_{a}$ and similarly in the rational setting we have
$Q_{nG} = Q_{nL} \circ Q_a$. Therefore we want to use  the
composite functor spectral sequence of Blanc and Stover in order to
approximate the $E^2$-term of our resolution spectral sequence (as in
Theorem \ref{thm:e2term} and Theorem \ref{thm:e2rational}). Let
$\mathcal{C}$ and $\mathcal{B}$ denote categories of universal graded
algebras such as the category of graded commutative algebras (over
$\Q$ or $\F_2$), the category of (restricted) $n$-Lie algebras or of
(restricted) $n$-Gerstenhaber algebras. Let $\mathcal{A}$ denote a
concrete category such as the category of graded vector spaces over a
field. Moreover let $T\colon \mathcal{C} \ra \mathcal{B}$ and $Z\colon
\mathcal{B} \ra \mathcal{A}$ be functors, then Blanc and Stover prove
the existence of the following spectral sequence.

\begin{thm}\cite[Theorem 4.4]{BS} \label{thm:bsss}
Suppose that $TF$ is $Z$-acyclic for every free $F$ in
$\mathcal{C}$. Then for every $C$ in $\mathcal{C}$ there is a
Grothendieck spectral sequence with
$$ E^2_{s,t} = (\mathbb{L}_s\bar{Z}_t)(\mathbb{L}_*T)C \Rightarrow
(\mathbb{L}_{s+t}(Z \circ T))C.$$
\end{thm}
The condition that $TF$ is $Z$-acyclic means that the left derived
functors $\mathbb{L}_*Z$ applied to $TF$ are trivial but in degree
zero where they are isomorphic to $ZTF$. The terms $\mathbb{L}_*\bar{Z}$ are
a certain extension of the derived functors of $Z$ to the category of
$\Pi$-$\mathcal{B}$-algebras, \ie, $\bar{Z}$ takes the homotopy
operations into account that live on the homotopy groups of every
simpicial $\mathcal{B}$-algebra.

If we unravel the notation in \cite[Theorem 4.4]{BS} then the $E^2$-term gives
$$ E^2_{s,t} = \pi_s \pi_t^i Z(B_{\bullet,\bullet})$$
where the notation is as follows: Let $Y_\bullet \ra C$ be a cofibrant
resolution of $C$ in the Quillen model category of simplicial objects
in $\mathcal{C}$. Then $B_{\bullet,\bullet}$ is a free resolution of
$TY_\bullet$ in the $E^2$-model category structure on the category of
bisimplicial objects in $\mathcal{B}$ \cite[5.10]{DKS},
\cite[4.1]{BS}. As we have various $E^2$-terms floating around, we
will call this model structure the DKS-model structure.  If
$B_{\bullet,\bullet}$ in bidegree $(t,s)$ is $B_{t,s}$ then $\pi_t^i$
is the $t$th homotopy group with respect to the first simplicial
direction and then $\pi_t^i B_{\bullet,\bullet}$ is a free simplicial resolution of
$\pi_t TY_\bullet$ by $\Pi$-$\mathcal{B}$-algebras.

\begin{prop}
\begin{itemize}
\item[]
\item
If the ground field is $\F_2$ and if we consider the sequence of functors
$$\xymatrix@1{{\rg1} \ar[rr]^{Q_a} & & {\rlie1} \ar[rr]^{Q_{\rlie1}} &
  & {\grf2}},$$
then for any $C \in \rg1$ the $E^2$-term of the composite functor spectral sequence simplifies to
$$ E^2_{s,t} = (\mathbb{L}_s(\bar{Q}_{\rlie1})_t)(\AQ_*(C|\F_2,\F_2)).$$
\item
For the sequence
$$\xymatrix@1{{\ng} \ar[rr]^{Q_a} & & {\nlie} \ar[rr]^{Q_{nL}} & & {\grvctq}}$$
over the rationals and for $C \in \ng$, the spectral sequence has $E^2$-term isomorphic to
$$ E_{s,t}^2 = (\mathbb{L}_s(\bar{Q}_{\nlie})_t)(\AQ_*(C|\Q,\Q)).$$

\end{itemize}
\end{prop}
\begin{proof}
There are two adjoint pairs, $(\rg1,V)$ and $(\ng,V)$,  where $V$
denotes the forgetful functor to the underlying category of graded
vector spaces. Associated to these are standard simplicial
resolutions for calculating $\mathbb{L}_*Q_a(C)$, namely
$(\rg1 \circ V)^{\bullet +1}(C)$ for characteristic two and $(\ng
\circ V)^{\bullet +1}(C)$ for characteristic zero. The free
$1$-restricted Gerstenhaber algebra generated by a graded vector
space $W_*$ is $S(1rL(W_*))$, the free graded commutative algebra generated by
the free $1$-restricted Lie algebra  on $W_*$. In particular the above
mentioned resolutions consist of free graded commutative
algebras and $Q_a(S(1rL(W_*)))$  is $1rL(W_*)$, which is
$Q_{1rL}$-acyclic. Since the derived functors of $Q_a$ compute Andr\'e-Quillen
homology we get a spectral sequence of the form above. Similar
arguments hold for $n$-Gerstenhaber algebras.
\end{proof}

\begin{rem} \label{rem:resolutions}
How can one calculate these $E^2$-terms? First one resolves $C$
simplicially by free Gerstenhaber algebras, $P_\bullet \ra C$, and takes
indecomposables, $Q_a(P_\bullet)$. This is now a simplicial object in
some category of Lie algebras ($n$-Lie or restricted $1$-Lie), thus
one has to find a free resolution of this object in the DKS-model structure
as explained in \cite[4.1.1]{BS}
\end{rem}

Over the rationals we want to compare $E_n$-homology of a commutative
algebra with $E_m$-homology for $n\neq m$. To this end, we first
compare the model category structures on the corresponding categories
of simplicial shifted Lie algebras.

Let us briefly recall the model category structure on simplicial
$n$-Lie algebras, $s\nlie$ \cite[II, \S 4]{Q}:
\begin{itemize}
\item
A map $f\colon \mathfrak{g}_\bullet \ra \mathfrak{g}'_\bullet$ is a
\emph{weak equivalence} if $U(f)$ is a weak equivalence of simplicial
$\Q$-vector spaces.
\item
Such a map is a \emph{fibration}, if the induced map $\bar{f}$
$$\xymatrix{
{\mathfrak{g}_\bullet} \ar@{.>}[rd]^{\bar{f}} \ar@/^3ex/[rrd]^{f}
\ar@/_2ex/[rdd]&{}&{}\\
{}&{\pi_0(\mathfrak{g}_\bullet) \times_{\pi_0(\mathfrak{g}'_\bullet)}
  \mathfrak{g}'_\bullet} \ar[r] \ar[d]& {\mathfrak{g}'_\bullet}
\ar[d]\\
{}&{\pi_0(\mathfrak{g}_\bullet)} \ar[r]^{\pi_0(f)}&{\pi_0(\mathfrak{g}'_\bullet)}
}$$
is surjective.
\item
A map is a \emph{cofibration} if it has the left lifting property with
respect to acyclic fibrations.
\end{itemize}
\begin{prop} \label{prop:mcat}
The model categories of simplicial $n$-Lie algebras and of simplicial
graded Lie algebras are Quillen equivalent.
\end{prop}
\begin{proof}
The equivalence $(\Sigma^n, \Sigma^{-n})$ between $n$-Lie algebras and
graded Lie algebras that we mention in  Remark \ref{rem:suspensions}
preserves and detects weak equivalences as these are given by
reference to the underlying simplicial graded vector spaces, hence a
morphism
$f\colon \mathfrak{g} \ra \Sigma^{-n}\mathfrak{g}'$ of $n$-Lie algebras is a
weak equivalence if and only if $\Sigma^n(f)\colon \Sigma^n
\mathfrak{g} \ra \mathfrak{g}'$ is one.
Similarly, suspension does not affect surjectivity of maps.
\end{proof}
What we actually need is to extend this Quillen equivalence to the
corresponding model categories of simplicial  Lie $\Pi$-algebras.

\begin{thm} \label{thm:quillenequ}
The model categories of simplicial $n$-Lie-$\Pi$ algebras and of
simplicial graded Lie $\Pi$-algebras are Quillen equivalent.
\end{thm}
\begin{proof}
We first show that the $n$-fold suspension and desuspension functors
pass to functors between (shifted)-$\Pi$-Lie
algebras. Let $\mathfrak{g}_*$ be an $n$-Lie algebra with $n$-Lie
homotopy operations. These operations are parametrized by
elements in
$$ [nL(\Q\mathbb{S}^m(\ell)_\bullet),nL(\bigoplus_{i=1}^N
\Q\mathbb{S}^{m_i}(\ell_i)_\bullet)]_{s\nlie}$$
where $\mathbb{S}^m(\ell)$ is the graded simplicial set that has the
simplicial $m$-sphere in degree $\ell$. If we consider
$\Sigma^n\mathfrak{g}_*$, then this inherits operations parametrized by
\begin{align*}
[\Sigma^n nL(\Q\mathbb{S}^m(\ell)_\bullet),\Sigma^n nL(\bigoplus_{i=1}^N
\Q\mathbb{S}^{m_i}(\ell_i)_\bullet)]_{s\lie} = &
[L(\Sigma^n\Q\mathbb{S}^m(\ell)_\bullet),L(\bigoplus_{i=1}^N
\Sigma^n\Q\mathbb{S}^{m_i}(\ell_i)_\bullet)]_{s\lie} \\
= & [L(\Q\mathbb{S}^m(\ell-n)_\bullet),L(\bigoplus_{i=1}^N
\Q\mathbb{S}^{m_i}(\ell_i-n)_\bullet)]_{s\lie}.
\end{align*}
Vice versa the $n$-fold desuspension of a graded Lie $\Pi$-algebra
inherits Lie  homotopy operations.

Weak equivalences and fibrations are again determined by the
underlying simplicial vector spaces and on this level (de)suspensions
just shift the internal grading.
\end{proof}

\begin{lem}\label{lem:aquiBSE2}
Let $A$ be an augmented $E$-algebra and $\ell \in \mathbb{Z}$. Then there
is a natural isomorphism
$$ \Sigma^{n-\ell} (\mathbb{L}_s \bar{Q}_{nL})_t(\AQ_*(H_*A|\Q;\Q)) \cong
(\mathbb{L}_s \bar{Q}_{\ell L})_t(\Sigma^{n-\ell} \AQ_*(H_*A|\Q;\Q)).$$
\end{lem}
\begin{proof}
Let $Y_{\bullet}$ be a resolution of $H_*A$ as a simplicial
augmented Gerstenhaber algebra such that $Y_{\bullet}$ is cofibrant as
a simplicial commutative algebra. As we explained above, we have
$$(\mathbb{L}_s
\bar{Q}_{nL})_t(\AQ_*(H_*A|\Q;\Q)) \cong \pi_s \pi_t^i Q_{nL}(B_{\bullet,
  \bullet})$$
where $B_{\bullet, \bullet}$ is a cofibrant replacement
of $Q_a Y_{\bullet}$ in the category of bisimplicial $n$-Lie algebras
with respect to the DKS-model structure. According to \cite[4.1]{BS}
the following conditions are sufficient for $B_{\bullet, \bullet}$ to
be cofibrant:
\begin{itemize}
\item For fixed external degree $s$ each $B_{\bullet, s}$ is
  homotopy equivalent to $nL(X[s]_{\bullet})$ as a simplicial $n$-Lie algebra,
  where $X[s]_{\bullet}$ is weakly equivalent to a sum of spheres
  $\bigoplus_i \mathbb{Q} S^{m_i}(r_i)_{\bullet}$.
\item The external degeneracies are induced by maps $X[s] \rightarrow
  X[s+1]$ which are inclusions of summands up to homotopy.
\end{itemize}
Blanc and Stover show as well that such a $B_{\bullet, \bullet}$ can
always be constructed.

Suppose that $B_{\bullet, \bullet}$ is a cofibrant replacement
fulfilling these conditions. Now consider $\Sigma^{n-\ell} B_{\bullet,
  \bullet}$. Regrading the internal nonsimplicial degree does not
affect the homotopy groups, so
$$ \pi_s \pi_t^i \Sigma^{n-\ell} B_{\bullet, \bullet} = \delta_{s,0}
\Sigma^{n-\ell} \pi_t Q_a Y_{\bullet} = \delta_{s,0} \pi_t
\Sigma^{n-\ell} Q_a Y_{\bullet}$$
and we find that $\Sigma^{n-\ell} B_{\bullet, \bullet}$ is a
resolution of $\Sigma^{n-\ell} Q_a Y_{\bullet}$. It is easy to
see that $\Sigma^{n-\ell} B_{\bullet, \bullet}$ is cofibrant as well:
We know that
$ \Sigma^{n-\ell} B_{\bullet, s} = \ell L(\Sigma^{n-\ell}
X[s]_{\bullet})$ and that suspending a simplicial sphere $\Q
S^{m_i}(r_i)$ internally just shifts $r_i$. Hence we can
compute $(\mathbb{L}_s \bar{Q}_{\ell L})_t(\Sigma^{n-\ell}
\AQ_*(H_*A|\Q;\Q))$ as the homotopy groups of $Q_{\ell L}\Sigma^{n-\ell}
B_{\bullet, \bullet}$. Exploiting the adjunction between $Q_{\ell L}$
and the functor which endows a graded $\mathbb{Q}$-vector space with a
trivial $\ell$-Lie algebra structure we obtain
$$Q_{\ell L} \circ \Sigma^{n-\ell} \cong \Sigma^{n-\ell} Q_{nL}.$$
Therefore the homotopy groups in question are the homotopy groups of
$$\Sigma^{n-\ell} Q_{nL} B_{\bullet, \bullet}$$
and the result follows.
\end{proof}

Since the composite functor spectral sequence is the spectral sequence
associated to the bisimplicial object given by applying the
indecomposables functor to a resolution with respect to the DKS model
structure it is clear that an isomorphism of resolutions yields a
morphism of spectral sequences. Note that deriving a functor followed
by a suspension equals suspending the derived functor.

\begin{cor} The above isomorphism is part of a morphism between the
  suspension of the Grothendieck spectral sequence
$$ \Sigma^{n-\ell} ( \mathbb{L}_s \bar{Q}_{nL})_t(\AQ_*(H_*A|\Q;\Q))
\Rightarrow \Sigma^{n-\ell}\mathbb{L}_{s+t} Q_{nG} H_*(\bar{A})$$
associated to $Q_{nL} \circ Q_a$ and the Grothendieck spectral sequence
$$  (\mathbb{L}_s \bar{Q}_{\ell L})_t(\Sigma^{n-\ell}\AQ_*(H_*A|\Q;\Q))
\Rightarrow \mathbb{L}_{s+t} (Q_{\ell L}\circ \Sigma^{n-\ell} \circ
Q_a) H_*(\bar{A}).$$
\end{cor}

\begin{rem}
We could deduce that the $E^2$-term of the spectral sequence
calculating $E_{n+1}$-homology is
$$ \mathbb{L}_p Q_{nG} (H_*(\bar{A})) \cong \Sigma^{\ell-n} \mathbb{L}_p
(Q_{\ell L} \circ \Sigma^{n-\ell} \circ Q_a)(H_*(\bar{A})),$$
but this is clear since
$$Q_{\ell L} \circ \Sigma^{n-\ell} \cong \Sigma^{n-\ell} Q_{nL}.$$
\end{rem}

Identifying  $n$-Lie-$\Pi$ algebra structures is
hard. Sometimes we can reduce the complexity of that task.
In characteristic zero, the Blanc-Stover spectral sequence simplifies
due to the following well-known result.

\begin{lem}
Let $k = \Q$. A Lie-$\Pi$ algebra
$\pi_*(\mathfrak{g}_{\bullet})$ is a
bigraded Lie algebra.
\end{lem}
\begin{proof}
The Lie-$\Pi$ structure on the homotopy groups of a simplicial graded
Lie algebra is the structure induced by elements in
$$ [ L(\Q S^n(k)), L(\bigoplus_{i=1}^N \Q S^{n_i}(k_i))]_{sL}
=\pi_n(L(\bigoplus_{i=1}^N \Q S^{n_i}(k_i)))_k. $$
Set $X= \bigoplus_{i=1}^N \Q S^{n_i}(k_i)$.
Interpreting the Lie operad as a constant simplicial operad in graded
vector spaces we find that we need to calculate
$\pi_n(\bigoplus_{j\geq 0}\mathit{Lie}(j)\otimes_{\Sigma_j}(X)^{\otimes
  j})_k$.

Since over $\Q$ every $\Sigma_j$-module is projective we see that this
is isomorphic to
$$\bigoplus_{j\geq 0} \bigoplus_{a+b=k}
\mathit{Lie}(j)_{a}\otimes_{\Sigma_j}
\bigoplus_{\stackrel{n_1+\ldots+n_j=n}{b_1+\ldots+b_j=b}}
\pi_{n_1}(X)_{b_1}\otimes \ldots\otimes \pi_{n_j}(X)_{b_j},$$
\ie,  the free bigraded Lie algebra on $N$ generators of degree
$(k_i,n_i)$, where we now consider the Lie operad as an operad in bigraded
modules concentrated in bidegree $(0,0)$. This yields that all
homotopy operations on $\pi_*(\mathfrak{g}_\bullet)$ are the ones
induced by the Lie structure of $\mathfrak{g}_\bullet$ via the
Eilenberg-Zilber map.
\end{proof}

We obtain an analogous result in finite characteristic.
\begin{lem} \label{lem:pi0}
Let $k$ be $\F_p$.
If $\mathfrak{g} = \pi_*X_\bullet$ is a restricted $\Pi$-Lie-algebra
that is concentrated
in  $\pi_0X_\bullet$, then the $\Pi$-Lie-algebra structure on
$\mathfrak{g}$ reduces to a restricted Lie-algebra structure.
\end{lem}
\begin{proof}
According to \cite[\S 3]{BS} the operations on $\pi_0X_\bullet$ are
parametrized by elements in the set of homotopy classes of simplicial
restricted Lie algebras
$$ [rL(\mathbb{S}^0(k)),rL(\bigoplus_{i=1}^N \mathbb{S}^0(k_i)]_{srL}.$$
Here, $\mathbb{S}^0(r)$ is the simplicial graded $\F_p$-vector space
that is $\F_p[r]$ in
every simplicial degree and $\F_p[r]$ is the graded $\F_p$-vector
space that is $\F_p$
concentrated in degree $r$.
As the
simplicial direction is constant in this case, the above set of
homotopy classes reduces to the set of homomorphisms of restricted
Lie-algebras
$$ rL(rL(\F_p[k]),rL(\bigoplus_{i=1}^N \F_p[k_i])) \cong
rL(\bigoplus_{i=1}^N \F_p[k_i])_k.$$
Thus we get the free restricted Lie-algebra on $N$ generators in degree $k$
and the operations reduce to a restricted
Lie-structure on  $\mathfrak{g}$.
\end{proof}

\part{Examples} \label{part:examples}

\section{$E_{n+1}$-homology of free graded commutative
algebras} \label{sec:freegraded}
In the following, we will consider free graded commutative algebras on
one generator. For the general case note that working over a field
ensures that $E_{n}$-homology of a tensor product of graded
commutative algebras can be computed from the $E_{n}$-homology of
the tensor factors: if $A_*,B_*$ are two graded commutative algebras,
then $E_{n}$-homology of $\overline{A_*\otimes B_*}$ can be
identified with Hochschild homology of order $n$ with coefficients in
the ground field $k$:
$$ H^{E_n}_*(\overline{A_*\otimes B_*}) \cong \HH_{*+n}^{[n]}(A_*
\otimes B_*;k)$$
which is defined as the homotopy groups of some simplicial set arising
as the evaluation of a certain $\Gamma$-module $\mathcal{L}(A_* \otimes
B_*;k)$ on the simplicial $n$-sphere,
$$\HH_{*+n}^{[n]}(A_* \otimes B_*;k) \cong \pi_{*+n}\mathcal{L}(A_* \otimes
B_*;k)(\mathbb{S}^n).$$
The latter is isomorphic to
$$  \pi_{*+n}(\mathcal{L}(A_*;k)(\mathbb{S}^n) \otimes
\mathcal{L}(B_*;k)(\mathbb{S}^n))$$
and hence the K\"unneth theorem expresses this in terms of tensor
products of Hochschild homology groups of order $n$. For more
background on Hochschild homology of order $n$ see \cite{P} or
\cite[p.~207]{LR}.

\subsection{Characteristic zero, $n \geq 1$}
Let $A=\Q[x]$ with the generator $x$ being of degree zero, thus $H_*A=A$.

We know that $E_{n+1}$-homology of the non-unital algebra
$\overline{\Q[x]}$ is isomorphic
to the shifted $\Q$-homology of $K(\Z,n+1)$ (see \cite{C} or
\cite{LR}):
$$H^{E_{n+1}}_*(\overline{\Q[x]}) \cong H_{*+n+1}(K(\Z,n+1),\Q). $$
We know that rationally the cohomology of $K(\Z,n+1)$ is an exterior
algebra on a generator in degree $n+1$ if $n+1$ is odd and is a polynomial
algebra on a generator of degree $n+1$ for even $n+1$ and thus by
dualizing we get the answer for $H_{*+n+1}(K(\Z,n+1),\Q)$.

For odd $n+1$ the polynomial algebra $\Q[x]$ is actually a free
$n$-Gerstenhaber algebra because the bracket $[x,x]$ has to be
trivial. Therefore the derived functors of $n$-Gerstenhaber
indecomposables are trivial but in degree zero where we obtain the
$\Q$-span of $x$ and thus
$$ H^{E_{n+1}}_*(\overline{\Q[x]}) \cong \begin{cases} \Q, & *=0, \\ 0, &
  *>0, \end{cases}$$
and this agrees with the above result.

For arbitrary degrees the result is essentially the same:
\begin{prop}
The free graded commutative algebra $S(x_j)$ on a generator $x_j$ of
degree $j\in \Z$ has as $E_{n+1}$-homology:
$$ H^{E_{n+1}}_*(\overline{S(x_j)}) \cong \Sigma^{-(n+1)} \overline{S(x_{j+n+1})}.$$
In particular we can identify $E_{n+1}$-homology of
$\overline{S(x_j)}$ with the shifted homology of the
Eilenberg-MacLane space
$K(\Z, j+n+1)$ if $j+n+1>0$.
\end{prop}
\begin{proof}
The $E^2$-term of the resolution spectral sequence can be identified
with
$$E^2_{s,q} \cong \mathrm{Tor}_{s+1}^{S(x_{n+j})}(\Q,\Q)_{q+n} \cong
S^{s+1}(x_{n+j}[1])_{q+n}.$$
The internal degree of $x_{n+j}[1]$ is still $n+j$ but for forming the
free graded commutative algebra, $S(x_{n+j}[1])$, $x_{n+j}[1]$ is
viewed as a generator in degree $n+j+1$. By $S^{s+1}$ we denote the
monomials of length $s+1$.
Powers of $x_{n+j}[1]$ are trivial if $n+j$ is even thus we get a
single contribution from $s=0$ of internal degree $n+j = q+n$.

In the odd case, we
obtain the condition that the
internal degree of $x_{n+j}[1]^{s+1}$ has to be $q+n$, thus
$(s+1)(n+j) = q+n$. Therefore we get that $s+q$ has to be $sn+sj+s+j$
as claimed.
\end{proof}

\begin{rem} \label{rem:shuffleforfreecom} The isomorphism
$$ H^{E_{n+1}}_*(\overline{S(x_j)}) \cong \Sigma^{-(n+1)} \overline{S(x_{j+n+1})}$$
is not only an additive isomorphism but an isomorphism of algebras if
we consider $H^{E_{n+1}}_*(\overline{S(x_j)}) $ equipped with the
multiplicative structure arising for example from computing
$H^{E_{n+1}}_*(\overline{S(x_j)}) $ via the iterated bar construction
$B^{n+1}$ as in \cite{F3}: If  $j+n+1$ is even and we denote the
suspension of an element $x$ by $sx$, generating cycles in
$B^{n+1}(\overline{S(x_j)})$ are of the form
$$s(s^n x_j) \otimes \ldots \otimes s(s^n x_j) \in (\Sigma
B^n(\overline{S(x_j)}))^{\otimes r}$$
with $s^n x_j \in \Sigma^n \overline{S(x_j)} \subset B^n
(\overline{S(x_j)})$.  One easily calculates that
$$(s(s^n x_j))^{\otimes r} \cdot (s(s^n x_j))^{\otimes s} = {r+s
  \choose r }( s(s^n x_j))^{\otimes s+r}$$
with respect to the shuffle product on $B^{n+1}(\overline{S(x_j)})$,
hence $H^{E_{n+1}}_*(\overline{S(x_j)}) $ is, up to suspension,
isomorphic to a polynomial algebra over $\Q$.
\end{rem}
\subsection{Characteristic two}
For a generator, $x_0$, in degree zero, we have that $E_2$-homology of
the non-unital polynomial algebra on $x_0$ over a field $k$ is (up to
a $2$-shift in degree) the homology with $k$-coefficients of $K(\Z,2)
= \mathbb{C}P^\infty$. It turns out that shifting the degree of the polynomial
generator down to degree minus one, trades the complex numbers for the
reals. More generally, we compute $E_2$-homology for every polynomial
algebra $\F_2[x_n]$ with a generator of degree $n \in \Z$. We always
assume that the $1$-restricted Lie structure on the polynomial algebra
is trivial. Note that the suspension of a $1$-restricted Lie algebra is a
restricted Lie algebra, similar to \ref{rem:suspensions}.
\begin{prop}
Up to a shift, $E_2$-homology of a polynomial algebra $\F_2[x_{-1}]$
is isomorphic to the homology of $\R P^\infty$:
$$ H^{E_2}_s(\overline{\F_2[x_{-1}]}) \cong H_{s+1}(K(\Z/2\Z,1);\F_2).$$
\end{prop}
\begin{proof}
According to \ref{rem:e2smoothchartwo} our $E^2$-term is given by
$$E^2_{p,q} = ((\mathbb{L}_p Q_{1rL})(\F_2 x_{-1}))_q.$$
Using the suspension, we obtain from the $\F_2$-analogue of
\ref{cor:shift} that
$$\Sigma(\mathbb{L}_pQ_{1rL})(\F_2 x_{-1})
\cong (\mathbb{L}_p Q_{rL})(\Sigma
\F_2 x_{-1}) \cong
(\mathbb{L}_pQ_{rL})(\F_2x_0)$$
where $x_0$ is a generator in degree zero.

As the restricted Lie-structure on
$\F_2x_0$ is trivial, its restricted enveloping algebra is
$$\mathfrak{U}_r(\F_2x_0) \cong \F_2[x_0]/x_0^2 \cong \F_2[C_2].$$

Therefore
$$ (\mathbb{L}_s Q_{rL})(\F_2 x_0)
\cong \mathrm{Tor}_{s+1}^{\F_2[C_2]}(\F_2,\F_2) \cong
H_{s+1}(C_2;\F_2) = H_{s+1}(\R P^\infty; \F_2).$$

Hence the $E^2$-term is concentrated in bidegrees
$(s+1,-1)$, thus $H_{s+1}(\R P^\infty; \F_2)$ is isomorphic to
$H^{E_2}_s(\overline{\F_2[x_{-1}]})$.
\end{proof}

In broader generality we can determine $E_2$-homology of $\F[x_j]$ for
arbitrary degree $j \in \Z$. Again varying the degree $j$ of $x_j$
results in a shift.

\begin{prop}
With respect to the multiplicative structure induced by the shuffle
product on the twofold bar construction we have
$$ H^{E_2}_*(\overline{{\F_2}[x_j]}) = \Sigma^{-2}
\overline{\Gamma_{\F_2}(x_{j+2})}.$$
where $\Gamma_{\F_2}(x_{j+2})$ is a divided power algebra on a
generator $x_{j+2}$ in degree $j+2$.
\end{prop}
\begin{proof}
As above we get that
$$\Sigma (\mathbb{L}_sQ_{1rL})(\F_2x_j) \cong
(\mathbb{L}_sQ_{rL})(\F_2x_{j+1})=
\mathrm{Tor}_{s+1}^{\mathfrak{U}(\F_2x_{j+1})}(\F_2,\F_2). $$
As the Lie-structure on $\F_2x_{j+1}$ is trivial, the enveloping
algebra is again a truncated polynomial algebra. We take the
periodic free resolution of $\F_2$ over $\F_2[x_{j+1}]/x_{j+1}^2$
$$ \ldots \ra \Sigma^{\ell(j+1)}\F_2[x_{j+1}]/x_{j+1}^2 \ra
\ldots \ra
\Sigma^{j+1}\F_2[x_{n+1}]/x_{j+1}^2 \ra \F_2[x_{j+1}]/x_{j+1}^2$$
where the maps are given by multiplication by $x_{j+1}$.
Tensoring this with $\F_2$ over $\F_2[x_{j+1}]/x_{j+1}^2$ gives a
chain complex with $\Sigma^{\ell(j+1)}\F_2$ in degree $\ell$ and
trivial differential, and therefore
$$\mathrm{Tor}_{s+1}^{\mathfrak{U}(\F_2x_{j+1})}(\F_2,\F_2)  \cong
\Sigma^{(s+1)(j+1)}\F_2.$$

Note that this
is concentrated in bidegree $(p,q)=(s,(s+1)(j+1)-1)$ and
contributes to total degree $p+q=(s+1)(j+2)-2$.

In order to determine the multiplicative structure on
$H^{E_2}_*(\overline{{\F_2}[x_j]})  \cong
H_*(\Sigma^{-2}B^2(\overline{{\F_2}[x_j]}))$ we note that just as in
\ref{rem:shuffleforfreecom} generating cycles in
$B^2(\overline{{\F_2}[x_j]})$ are of the form
$$s (sx_j) \otimes \ldots \otimes s(sx_j).$$
Again we see that
$$(s(s x_j))^{\otimes r} \cdot (s(s x_j))^{\otimes s} = {r+s \choose r
} (s(s x_j))^{\otimes s+r},$$
hence $\Sigma^2 H^{E_2}_*(\overline{{\F_2}[x_j]})$ is a divided power
algebra with $\gamma_r(s(s x_j))= (s(s x_j))^{\otimes r}$.
\end{proof}

\section{Reduced Hochschild cochains}
\label{sec:hhcochains}
Normalized Hochschild cochains $C^*(A,A)$ of an associative
$k$-algebra $A$ constitute an
algebra over an unreduced $E_2$-operad \cite{MS}. The induced
$1$-Gerstenhaber structure on
Hochschild cohomology was already described in
\cite{Ge}. Unfortunately the augmentation inherited from
$C^0(A,A) \cong
A \ra k$ is not compatible with the $E_2$-structure: Applying the
brace operations as described by McClure and Smith to cochains in the
augmentation kernel might yield the unit element in $C^0(A,A)$.
Hence we are led to consider the following variant: Denote by
$\bar{C}^*(A,A)$ the cochain complex with
$$ \bar{C}^k(A,A)= \begin{cases} C^*(\bar{A},\bar{A}), k >0,\\
C^0(\bar{A}, A), k=0,
\end{cases}$$
and set the braces to be zero whenever one of the arguments is in $k
\subset \bar{C}^0(A,A)$.  One easily checks that this is an
$E_2$-algebra and that the augmentation respects the $E_2$-structure.

The augmentation ideal of these \emph{reduced
  Hochschild cochains} is $C^*(\bar{A}, \bar{A})$. \ie,  the
normalized cochain complex computing Hochschild cohomology of $A$ with
coefficients in $\bar{A}$.

Note  that we consider Hochschild cochains, so with respect
to cohomological grading the Lie bracket on cohomology is of degree
$-1$.  In the
following we consider Hochschild cochains as a non-positively graded chain
complex, so that we get an ordinary $1$-restricted Gerstenhaber structure
on the homology.

\subsection{ Hochschild cochains for $k[x]$}

The following is a standard result.
\begin{lem} \label{lem:hochschild-tv}
For $A=k[x]$ reduced Hochschild cohomology is
$$ \HH^0(k[x],\overline{k[x]}) =  \HH^1(k[x],\overline{k[x]}) \cong
\overline{k[x]}.$$
\end{lem}
In order to exploit our spectral sequence, we have to understand
the induced structure on Hochschild cohomology. Hochschild cohomology is concentrated in degrees zero and one, so the
multiplication (cup-product) gives rise to a square-zero
extension.

A derivation $f \in \Der(k[x],\overline{k[x]})$ can
be identified with $f(x)$. We denote by $p'$ the formal derivative of
a polynomial $p \in k[x]$.
\begin{lem} \label{lem:lieonhh}
The Lie bracket
$$\HH^1(k[x],\overline{k[x]}) \otimes \HH^1(k[x],\overline{k[x]}) \ra
\HH^1(k[x],\overline{k[x]})$$
is given by the usual Lie bracket of derivations, \ie,
$$[f,g] = g \circ f - f  \circ g = g'f - f'g,$$
whereas for $\alpha \in \overline{k[x]}\cong \HH^0(k[x],
\overline{k[x]})$ and $f \in \HH^1
(k[x],\overline{k[x]})$ the bracket is given by
$$ [f, \alpha]:=f(\alpha) = \alpha'f.$$
For $k=\mathbb{F}_2$ the restriction $\xi\colon
\HH^1(\F_2[x],\overline{\F_2[x]}) \ra \HH^1(\F_2[x],\overline{\F_2[x]})$ sends
$f$ to $f'f$.
\end{lem}
\begin{proof}
The first facts can be found in \cite{Ge}. A direct calculation shows
that $\xi$ satisfies the properties of a restriction. To this end note
that second derivatives of polynomials in characteristic two
vanish. We also know that we actually have a ($1$-restricted)
Gerstenhaber structure on $\HH^*(\F_2[x], \overline{\F_2[x]})$, so the
restriction is
determined by $\xi(x)$. Then
$$ \xi(x)'g + \xi(x)g' = [\xi(x),g] = [x,[x,g]] = g + xg', \text{  for all } g$$
implies that $\xi(x) = x$, so our choice of $\xi$ is a unique
restriction on $\HH^*(\F_2[x],\overline{\F_2[x]})$.
\end{proof}
Note that the Lie-bracket is trivial on $\HH^0 \times \HH^0$ and so
is the restriction on $\HH^0$. In particular, $\HH^*$ is far from
being a free (restricted) $1$-Gerstenhaber algebra.

Let $k=\Q$. The Hochschild cohomology groups in question are
 $$ \HH^{*}(\Q[x],\overline{\Q[x]})_+ = \Q[x_0]\otimes\Lambda[y_{-1}],$$
 where $x_0 $ is identified with $x \in
 \HH^{0}(\Q[x],\overline{\Q[x]})$  and $y_{-1} $ with $x \in
 \HH^{1}(\Q[x],\overline{\Q[x]})$. The $1$-Lie structure described in
 Lemma \ref{lem:lieonhh} corresponds to setting $[y_{-1}, x_0]=x_0$,
 all other brackets can then be determined by using the Poisson
 relation. Hence applying  \ref{lem:1gres} yields the following.

\begin{thm} \label{thm:hhfree-onedim}
For $\Q[x]$ the $E_2$-homology of the reduced Hochschild
cochains on $\Q[x]$ is concentrated in degree $ -1$ with
$$ H_{-1}^{E_2}(\bar{C}^{*}(\Q[x], \Q[x]))\cong \Q.$$
\end{thm}
\begin{proof}
According to \ref{lem:1gres} the $E^2$-page of the spectral sequence
we consider is given by
$$E^2_{p,q} \cong (\mathbb{L}_pQ_{1L}(W))_q$$
with $1$-Lie structure on $W=\Q \langle x_0, y_{-1}\rangle$ given by
$[y_{-1}, x_0]=x_0$. Using the equivalence between the category of
$1$-Lie algebras and the category of Lie algebras this yields
$$E^2_{p,q} \cong (\Sigma^{-1} \mathrm{Tor}_{p+1}^{\mathfrak{U}(\Sigma W)}(\Q,\Q))_q,$$
where now $\Sigma W = \Q \langle x_1, y_0 \rangle$.
We use the resolution for graded Lie algebras given by May in \cite{May66}: Set
$$P_i:= \begin{cases}
\mathfrak{U}(\Sigma W), & i=0,\\
\Q \langle a^{(i)}_{i-1}, b^{(i)}_i  \rangle \otimes \mathfrak{U}(\Sigma W), &i>0,
\end{cases}
$$
with lower indices indicating the internal degree of the elements.
Define $P_1 \ra P_0$ by $a^{(1)}_0 \otimes 1 \mapsto y_0$ and
$b^{(1)}_1 \otimes 1 \mapsto x_1$. Define
$P_{i} \ra P_{i-1}$ by $a^{(i)}_{i-1}\otimes 1 \mapsto
b^{(i-1)}_{i-1}\otimes y_0 + (i-1)b^{(i-1)}_{i-1}\otimes 1-
a^{(i-1)}_{i-2}\otimes x_1$
and $b^{(i)}_{i} \otimes 1 \mapsto b^{(i-1)}_{i-1}\otimes x_1$ for $i>1$.
This is a $\mathfrak{U}(\Sigma W)$-free resolution of $\Q$, and hence
$E^2_{p,q}$ vanishes expect for $p=0$ and $q=-1$.
\end{proof}

\subsection{$\bar{C}^*(TV,TV)$ for $V$ a $\Q$-vector space of
  dimension at least two.}
Let $V$ be a fixed $\Q$-vector space of dimension at least $2$. After
adding a unit element, Hochschild cohomology of $TV$
with coefficients in $\bar{T}V$ can be identified as the
square-zero extension
$$ \HH^*(TV,\bar{T}V)_+ \cong \Q \rtimes M(-1)$$
where $M(-1) = \HH^{1}(TV,\bar{T}V) = \Der(\bar{T}V) / \lbrace
\text{inner derivations} \rbrace$ is concentrated in degree minus one
and $\Q $ is in degree zero. The first summand in the Hodge
decomposition of Hochschild homology is isomorphic to Harrison
homology which is
Andr\'e-Quillen homology up to a shift in degree. This allows us to
compute the input for the Blanc-Stover spectral sequence:
\begin{prop} \label{prop:aqctv}
$$ \AQ_*(\HH^*(TV,\bar{T}V)_+|\Q;\Q) \cong \HH^{(1)}_{*+1}(\Q \rtimes
M(-1);\Q)$$ and $ \HH_*^{(1)}(\Q \rtimes
M(-1);\Q)$ is additively isomorphic to the free graded Lie-algebra
generated by the graded vector space $M(-1)$.
\end{prop}
\begin{proof}
The first statement is a general fact about the relationship between
Andr\'e-Quillen homology, Harrison homology and Hochschild homology in
characteristic zero (see for instance \cite[4.5.13]{Loday}). Hochschild
homology of the graded square zero extension $\Q \rtimes M(-1)$ with
coefficients in $\Q$ is the homology of the complex
$$ \xymatrix@1{ {\ldots} \ar[r]^(0.3){b} & {\Q \otimes_\Q (\Q\rtimes
  M(-1))^{\otimes n}} \ar[r]^{b} & {\Q \otimes_\Q (\Q\rtimes
  M(-1))^{\otimes (n-1)}} \ar[r]^(0.8)b & {\ldots}}$$
where $\Q \otimes_\Q (\Q\rtimes M(-1))^{\otimes n}$ is in degree $n$
and the boundary map $b$ is an alternating sum of multiplication and
augmentation maps. This chain complex is the associated chain complex
of a simplicial graded $\Q$-vector space and its non-degenerate part in
degree $n$ is isomorphic to $M(-1)^{\otimes n}$. On that part, the
boundary is trivial and thus we get that
$$ \HH_*(\Q \rtimes M(-1);\Q) \cong T(M(-1))$$
where the $n$th homology group corresponds to tensors of length $n$.

The Hodge decomposition is given by an action of Eulerian idempotents on the
Hochschild chains and homology groups. The first idempotent,
$e_n^{(1)}\colon \HH_n(\Q \rtimes M(-1);\Q) \ra \HH_n(\Q \rtimes
M(-1);\Q)$ splits off Harrison homology. Reutenauer showed in
\cite[(2.2),(2.4)]{R} that the image of this idempotent applied to a tensor
algebra, $TW$, is precisely the free Lie algebra, $LW$.
\end{proof}
Note that in addition to the free graded Lie structure on Andr\'e-Quillen
homology we have the internal non-trivial (and non-free) $1$-Lie
structure coming from the first Hochschild cohomology group $M(-1)$.

With respect to this $1$-Lie structure
$\AQ_*(\HH^*(TV,\bar{T}V)_+|\Q;\Q)$ consists of the Lie subalgebra
$\AQ_0(\HH^*(TV,\bar{T}V)_+|\Q;\Q)= \HH^*(TV, \bar{T}V)$
and the ideal $\AQ_{*\geq 1}(\HH^*(TV,\bar{T}V)_+|\Q;\Q)$. In
 particular $$\mathbb{L}_s Q_{1L}   \HH^*(TV, \bar{T}V) =
 (\mathbb{L}_s Q_{1L})_0  \AQ_*(\HH^*(TV,\bar{T}V)_+|\Q;\Q).$$ Hence
 we can identify certain elements in the Blanc-Stover spectral
 sequence in the case $V =\Q \lbrace x,y \rbrace$.

\begin{lem} The Lie homology of $\HH^*(T\Q \lbrace x,y \rbrace,
  \bar{T}\Q \lbrace x,y \rbrace) $ is
$$\mathbb{L}_s Q_{1L}   \HH^*(T\Q \lbrace x,y \rbrace, \bar{T}\Q
\lbrace x,y \rbrace)  = \begin{cases}
\Q, & s=0,1,3,\\
0 & \text{else}.
\end{cases}$$
\end{lem}
\begin{proof}
We know that
$$ \HH^*(T\Q \lbrace x,y \rbrace, \bar{T}\Q \lbrace x,y \rbrace)  \cong
\Der (\bar{T}\Q \lbrace x,y \rbrace) )/ \lbrace \text{inner
  derivations} \rbrace$$
as a $1$-Lie algebra concentrated in internal degree $-1$. Consider
the derivations $D_{x,v}$ and $D_{y,w}$ defined by
$$D_{x,v}(x) = v, D_{x,v}(y)=0, \quad D_{y,w}(x) = 0, D_{y,w}(y) =w.$$
for $v,w \in \bar{T}\Q\lbrace x,y \rbrace$.
These form a basis of $\Der(\bar{T}\Q \lbrace x,y \rbrace)$ as a
vector space and are eigenvectors with respect to $[-, D_{x,x}]$ as
well as $[-,D_{y,y}]$.
Observe also that a typical inner derivation is of the form
$D_{x,vx}-D_{x,xv} + D_{y,vy} -D_{y,yv}$ and hence an eigenvector as
well. In particular $\Der (\bar{T}\Q \lbrace x,y \rbrace) )/ \lbrace
\text{inner derivations}\rbrace$ splits into eigenspaces with respect
to $[-, D_{x,x}]$ and $[-,D_{y,y}]$. Hence we can apply
\cite[1.5.2]{Fu}. Since possible eigenvalues are limited we see
that the Lie homology of
$\Der (\bar{T}\Q \lbrace x,y \rbrace) )/ \lbrace \text{inner
  derivations} \rbrace$ is the homology of the complex
$$\xymatrix{ 0 \ar[r] & \Q \lbrace D_{x,x} \wedge D_{y,y} \wedge
  D_{x,y} \wedge D_{y,x}\rbrace \ar[r] &
\Q \lbrace D_{x,x} \wedge D_{x,y} \wedge D_{y,x} , D_{y,y}
\wedge D_{x,y} \wedge D_{y,x} \rbrace \ar`r[d]`[l]`[lld]`[dl][dl] \\ &
 \Q \lbrace D_{x,x} \wedge D_{y,y}, D_{x,y} \wedge D_{y,x}
\rbrace\ar[r] & \Q \lbrace D_{x,x}, D_{y,y} \rbrace \ar`r[d]`[l]`[lld]`[dl][dl] \\
& \Q \lbrace D_{x,x}, D_{y,y} \rbrace \ar[r] & 0} $$
endowed with the usual differential of the Chevalley-Eilenberg
complex. Hence the claim follows.
 \end{proof}

\begin{prop} The $E^2$-page of the Blanc-Stover spectral sequence has
$$E^2_{s,0} = \begin{cases}
\Q & \text{in internal degree $-1$ for s=0},\\
\Q & \text{in internal degree $-2$ for s=1},\\
\Q & \text{in internal degree $-4$ for s=3},\\
0, & \text{else.}
\end{cases}$$
\end{prop}
\begin{rem}
The generators for $s=0$ and $s=1$  in the
Blanc-Stover spectral sequence are permanent cycles and they cannot be
boundaries for degree
reasons. Therefore they give rise to  permanent
cycles $x_{0,-1}$ and $x_{1,-2}$ in the resolution spectral
sequence. If the reduced Hochschild cochains on $T\Q\lbrace x,y
\rbrace$ were free as an $E_2$-algebra,
then we would get something of rank $2$ as $E_2$-homology and this
could correspond to these two survivors, but a priori $x_{0,-1}$ and
$x_{1,-2}$ could be hit by differentials starting on elements in
bidegree $(r,-r)$ for some $r \geq 2$.

\end{rem}
\subsection{Group algebras}
Let $G$ be a discrete group and $k$ be a field. There is an
identification of Hochschild cohomology of the group algebra $k[G]$
with group cohomology
$$ \HH^*(k[G],\overline{k[G]}) \cong H^*(G;\overline{k[G]}^c),$$
where $\overline{k[G]}^c$ denotes $\overline{k[G]}$ with the
$k[G]$-action being induced by the conjugation action of $G$ on $G$.
We will consider cases where Hochschild cohomology results in an
\'etale algebra, so we need the following result.

\begin{lem} \label{lem:etale}
Let $k$ be a field of characteristic zero or two and let $A$ be an
augmented \'etale $k$-algebra. Then $\bar{A}$ has trivial derived
$1$-(restricted) Gerstenhaber indecomposables and trivial $E_2$-homology.
\end{lem}
\begin{proof}
As explained in \ref{rem:resolutions}, let $P_\bullet \ra A$ be a
simplicial resolution of $A$  by
free $1$-(restricted) Gerstenhaber algebras.  As $A$ is
\'etale, it has
trivial indecomposables and
$Q_a(P_\bullet)$ has trivial homotopy groups in all degrees. Therefore the
constant bisimplicial  $1$-(restricted) Lie algebra which is zero in
all bidegrees is a valid resolution of $Q_a(P_\bullet)$.
Application of the composite functor spectral sequence yields the result.
\end{proof}

Hochschild cochains on some group algebras have trivial $2$-fold
algebraic delooping:
\begin{prop} \label{prop:etalegrings}
Let $G$ be a finite group. If
\begin{enumerate}
\item
either $G$ is abelian, the order of $G$ is odd and $k= \F_2$,
\item
or if $k$ is algebraically closed and of characteristic two and the
order of $G$ is odd,
\item
or if $k$ is algebraically closed and of characteristic zero,
\end{enumerate}
 then
$$H^{E_2}_*(\bar{C}^*(k[G],k[G])) = 0, \text{ for all }
* \geq 0.$$
\end{prop}
\begin{proof}
If $G$ is finite and if the characteristic of $k$ is prime to $|G|$ or
the characteristic is zero, then $H^*(G,\overline{k[G]}^c)\cong
H^0(G,\overline{k[G]}^c) =
(\overline{k[G]}^c)^G \cong Z(\overline{k[G]})$. The multiplication
induced by the $E_2$-action is the usual one. In the first case this
center is
$\overline{\F_2[G]}$. As $G$ is finite abelian, it suffices to consider the case
$C_{p^r}$ for an odd prime $p$. But $\F_2[C_{p^r}]$ is
\'etale over $\F_2$.

In the last two cases $k[G]$ is isomorphic to a product of matrix
rings (Wedderburn) and hence the center
is $Z(k[G]) \cong \prod_r k$, where $r$ is the number of conjugacy
classes of $G$. This is again an \'etale $k$-algebra.
\end{proof}

\section{On the Hodge decomposition for higher order
Hochschild homology} \label{sec:hodge}
Over the rationals the operad $E_{n}$ is formal, \ie, there is a
quasi-isomorphism between $E_{n}$ and the operad of $(n-1)$-Gerstenhaber
algebras (see \cite{LV} for a nice overview on formality). As every
$\Q[\Sigma_r]$-module $G_{n-1}(r)$ is projective, this
quasi-isomorphism induces an isomorphism of operadic homology theories
between $E_{n}$-homology and $G_{n-1}$-homology. As a consequence, our
resolution spectral sequence has to collapse at the $E^2$-term and we
obtain
\begin{equation*}
\bigoplus_{p+q=\ell} (\mathbb{L}_pQ_{(n-1)G}(\bar{A}))_q \cong
H^{E_n}_\ell(\bar{A}).
\end{equation*}

For an augmented commutative $\Q$-algebra $A$, we can identify
$E_n$-homology with Hochschild homology of order $n$:
\begin{equation*}
H^{E_n}_*(\bar{A}) \cong \HH^{[n]}_{*+n}(A,\Q).
\end{equation*}
The latter groups possess a Hodge decomposition \cite[Proposition
5.2]{P}. For odd $n$ the Hodge summands of Hochschild homology of
order $n$ are a re-indexed version of the Hodge summands for ordinary
Hochschild homology:
\begin{equation*}
\HH^{[n]}_{\ell +n}(A;\Q) = \bigoplus_{i+nj=\ell+n} \HH^{(j)}_{i+j}(A;\Q).
\end{equation*}
However, for even $n$ the summands are only described in terms of
functor homology:
\begin{equation*}
\HH^{[n]}_{\ell +n}(A;\Q) = \bigoplus_{i+nj=\ell+n}
\mathrm{Tor}_i^\Gamma(\theta^j, \mathcal{L}(A,\Q)).
\end{equation*}
For $j=1$ the terms consist of Andr\'e-Quillen homology:
$$\mathrm{Tor}_i^\Gamma(\theta^1, \mathcal{L}(A,\Q)) \cong
\AQ_i(A|\Q;\Q).$$
For $i=0$ one obtains $\theta^j \otimes_\Gamma \mathcal{L}(A,\Q) \cong
\Q \otimes_A  \mathrm{Sym}^j_A(\Omega^1_{A|\Q})$, the $j$-th symmetric
power generated by the module of K\"ahler differentials.

\begin{thm} \label{thm:hodge}
Let $A$ be a commutative augmented $\Q$-algebra.
For all $\ell, k \geq 1$ and $m \geq 0$:
\begin{itemize}
\item
$$ \HH^{(\ell)}_{m+1}(A;\Q) \cong (\mathbb{L}_m
Q_{2kG}\bar{A})_{(\ell-1)2k}.$$
\item
$$ \mathrm{Tor}^\Gamma_{m- \ell+1}(\theta^{\ell}, \mathcal{L}(A;\Q))
\cong (\mathbb{L}_mQ_{(2k-1)G}\bar{A})_{(\ell-1)(2k-1)}.$$
\end{itemize}
Thus the Hodge summands of higher order Hochschild homology can be
identified with Gerstenhaber homology groups.

\end{thm}
Of course, the convention is that negatively indexed Tor-groups vanish.

Note, that the case $\ell = 1$ comes for free:
The first Hodge summand is Andr\'e-Quillen homology,
$$\HH^{(1)}_{m+1}(A;\Q) \cong \AQ_m(A|\Q;\Q) \cong
\mathrm{Tor}_m^\Gamma(\theta^1, \mathcal{L}(A;\Q))$$
and this in turn is
$\mathbb{L}_mQ_{kG}(\bar{A})_0$ for all $m \geq  0$, $k \geq 1$.

For alternative approaches to the Hodge decomposition of higher order
Hochschild homology see \cite{bauer,Gi}.

In order to prove Theorem \ref{thm:hodge} we need a stability result. For
the remainder of this section $A\rightarrow \Q$ is an augmented commutative
$\Q$-algebra.

\begin{lem} \label{lem:stability}
The derived functors of Gerstenhaber indecomposables are stable in the
following sense:
$$(\mathbb{L}_mQ_{nG}\bar{A})_{qn} \cong
(\mathbb{L}_mQ_{(n+2)G}\bar{A})_{q(n+2)}. $$
\end{lem}
\begin{proof}
We consider the standard resolution that calculates
$(\mathbb{L}_mQ_{nG}\bar{A})$. In simplicial degree $\ell$ and
internal degree $r$ this is
$(nG)^{\ell+1}(\bar{A})_r$. This resolution is concentrated in degrees
of the form $r=qn$ because iterated $n$-Lie brackets on degree zero
elements are concentrated in these degrees. We can identify
the terms $(nG)^{\ell+1}(\bar{A})_{qn}$ with the terms
$((n+2)G)^{\ell+1}(\bar{A})_{q(n+2)}$ where we just exchange $n$-Lie
brackets by $(n+2)$-Lie brackets and adjust the internal degrees.

This yields an isomorphism of resolutions and hence an isomorphism on
the corresponding homology groups.

\end{proof}
\begin{rem}
Note that there is no stability result when one passes from $n$ to
$n+1$: Take for
instance $A = \Q[x]$. For even $n$ this is a free $n$-Gerstenhaber
algebra but for odd $n$ it is not.
\end{rem}
\begin{proof}[Proof of Theorem \ref{thm:hodge}]
As the claim is clear for $\ell =1$, we do an induction on the label
of the Hodge summands. We start with the $2k$-Gerstenhaber case. Thus
assume that we know the claim for all
Hodge summands $\HH^{(j)}_p$ for all $1 \leq j \leq \ell$ and all $p
\geq 1$. Lemma \ref{lem:stability} allows us to choose $k$ such that
$1 \leq m < 2k$ and to consider
$H^{E_{2k+1}}_{m+2k\ell}(\bar{A})$:
$$ \bigoplus_{p+q=m+2k\ell} \mathbb{L}_pQ_{2kG}(\bar{A})_q \cong
H^{E_{2k+1}}_{m+2k\ell}(\bar{A})
\cong \bigoplus_{i+j(2k+1)=m+2k(\ell+1)+1} \HH^{(j)}_{i+j}(A;\Q).$$
The summands $\mathbb{L}_{m+2k(\ell-r)}Q_{2kG}(\bar{A})_{2kr}$ for $r
< \ell$ are already identified with Hodge summands. The remaining
non-trivial summand in $H^{E_{2k+1}}_{m+2k\ell}(\bar{A})$ is
$\mathbb{L}_{m}Q_{2kG}(\bar{A})_{2k\ell}$ and in the Hodge
decomposition we still have the summand for
$j = \ell +1$. In this case
$$ i + (\ell+1)(2k+1) = m+2k(\ell+1)+1.$$
Hence $i=m+1-(\ell+1)$ and $i+j=m+1$.

For the $(2k-1)$-Gerstenhaber case the argument is similar, but the
degree count is different: For $j=\ell+1$
we get
$$ i+(\ell+1)2(k-2) = m-\ell + 2(k-2)(\ell+1)$$
and thus $i = m-\ell$.
\end{proof}
\begin{rem}
\emph{A posteriori} Theorem \ref{thm:hodge} yields a description of
derived functors of $2k$-Gersten\-ha\-ber algebras in terms of (higher)
Andr\'e-Quillen homology: A classical spectral sequence argument
allows an identification of $\HH^{(\ell)}_{m+1}(A,\Q)$ with
$D^{(\ell)}_{m+1-\ell}(A;\Q)$ \cite[3.5.8,4.5.13]{Loday} which in turn is
$H_{m+1-\ell}((\Lambda^\ell_{P_*} \Omega^1_{P_*|\Q}) \otimes_{P_*} \Q)) \cong
H_{m+1-\ell}(\Lambda^\ell (\Omega^1_{P_*|\Q} \otimes_{P_*} \Q))$. Here
$P_*$ is a free simplicial resolution of $A$ in commutative
$\Q$-algebras, for instance $P_t = (SI)^{\circ (t+1)}(A)$. Thus we obtain
$$  H_{m+1-\ell}(\Lambda^\ell (\Omega^1_{P_*|\Q} \otimes_{P_*} \Q))
\cong (\mathbb{L}_mQ_{2kG}\bar{A})_{(\ell -1)2k}.$$
\end{rem}
We show in the following that the identification of the Hodge summands
follows independently from an easy spectral sequence argument. We are
also able to prove an analoguous result for the even case:

\begin{thm} \label{thm:hodge-aq}
For every augmented commutative $\Q$-algebra $A$ we can identify the Hodge
summands of Hochschild homology of order $2k$ for $k \geq 1$ as
\begin{equation*}
\mathrm{Tor}_{m+1-\ell}^\Gamma(\theta^\ell, \mathcal{L}(A;\Q)) \cong
(\mathbb{L}_mQ_{(2k-1)}\bar{A})_{(2k-1)(\ell-1)} \cong
H_{m-\ell+1}(\mathrm{Sym}^\ell (\Omega^1_{P_*|\Q} \otimes_{P_*} \Q)).
\end{equation*}
We also recover the identification for Hodge summands of Hochschild
homology of odd order:
$$ \HH_{m+1}^{(\ell)}(A;\Q) \cong
\mathbb{L}_mQ_{2kG}(\bar{A})_{2k(\ell-1)} \cong
H_{m-\ell+1}(\Lambda^{\ell} (\Omega^1_{P_*|\Q} \otimes_{P_*} \Q)).$$
\end{thm}
\begin{rem}
The functor homology terms $\mathrm{Tor}_*(\theta^\ell,
\mathcal{L}(A;\Q))$ also describe the homology of the
$\ell$th homogeneous layer in the Taylor tower of the $\Gamma$-module
$\mathcal{L}(A;\Q)$, $D_\ell(\mathcal{L}(A;\Q))[1]$,
\cite[Proposition 4.7]{Ri}:
$$ H_*(D_ \ell(\mathcal{L}(A;\Q))[1]) \cong \mathrm{Tor}_*(\theta^\ell,
\mathcal{L}(A;\Q)).$$
Thus our results identifies these homology groups with derived
functors of $n$-Gerstenhaber indecomposables for odd $n$ and with the
homology of the $\ell$th symmetric power of the module of derived K\"ahler
differentials.
\end{rem}

\begin{proof}[Proof of Theorem \ref{thm:hodge-aq}]
Let $D_{*,*}$ be the bicomplex with $D_{r,s} = (nG)^{\circ
  (r+1)}((SI)^{\circ (s+1)}(A))$. Taking $n$-Gerstenhaber indecomposables
yields another bicomplex $C_{*,*}$ with $C_{r,s}= Q_{nG}(D_{r,s})
\cong (nG)^{\circ (r)}((SI)^{\circ (s+1)}(A))$:
$$\xymatrix{
{\vdots} \ar[d] & {\vdots} \ar[d] & {\vdots} \ar[d] & \\
{(SI)^{\circ (3)}(A)} \ar[d] & \ar[l] {(nG)((SI)^{\circ (3)}(A))} \ar[d] & \ar[l]
{(nG)^{\circ (2)}((SI)^{\circ (3)}(A))} \ar[d] & \ar[l] \ldots \\
{(SI)^{\circ (2)}(A)} \ar[d] & \ar[l] {(nG)((SI)^{\circ (2)}(A))} \ar[d] & \ar[l]
{(nG)^{\circ (2)}((SI)^{\circ (2)}(A))} \ar[d] & \ar[l] \ldots \\
{(SI)(A)} & \ar[l] {(nG)((SI)(A))} & \ar[l] {(nG)^{\circ (2)}((SI)(A))} &
\ar[l] \ldots
}$$
Taking vertical homology, $H^v_*$,  first and then horizontal
homology, $H^h_*$, gives
$$ H^h_r(H^v_s(C_{*,*})) \cong \mathbb{L}_rQ_{nG}(\bar{A})$$
concentrated in the $(s=0)$-line: the vertical homology groups
are trivial but for $s=0$ because $(SI)^{\circ (\bullet + 1)}(A)$ is a
resolution of $A$.

Switching the roles of vertical and horizontal homology gives
$$ H^v_r(H^h_s(C_{*,*})) \cong H_r\mathbb{L}_sQ_{nG}(SI)^{\circ (\bullet +
  1)}(A).$$
We know by Corollary \ref{cor:e2smooth} that
$\mathbb{L}_sQ_{nG}(SI)^{\circ (\bullet + 1)}(A)$ is
isomorphic to $\mathbb{L}_sQ_{nL}(SI)^{\circ (\bullet)}(A)$ and using the
suspension correspondence between $n$-Lie algebras and graded Lie
algebras we get
$$  \mathbb{L}_sQ_{nG}(SI)^{\circ (\bullet + 1)}(A) \cong
\Sigma^{-n}\mathbb{L}_sQ_{L}\Sigma^n (SI)^{\circ (\bullet)}(A).$$
Since $\Sigma^n (SI)^{\circ (\bullet)}(A)$ carries a trivial Lie
structure we can identify these groups as
$$ \Sigma^{-n}\mathrm{Tor}_{s+1}^{\mathfrak{U}(\Sigma^n(SI)^{\circ
    (\bullet)}(A))}(\Q,\Q) \cong
\Sigma^{-n}S^{s+1}(\Sigma^n(SI)^{\circ (\bullet)} A[1]).$$
Recall that $\Sigma^n(SI)^{\circ (\bullet)} A[1]$ is still concentrated in
internal degree $n$, but for the free graded commutative algebra
generated by it, $S(\Sigma^n(SI)^{\circ (\bullet)} A[1])$,
we consider its elements as being of degree $n+1$, thus the total
internal degree of elements in $\Sigma^{-n}S^{s+1}(\Sigma^n(SI)^{\circ
  (\bullet)} A[1])$ is $sn$.

For $n = 2k$ we therefore obtain
$$ \HH_{m+1}^{(\ell)}(A;\Q) \cong
\mathbb{L}_mQ_{2kG}(\bar{A})_{2k(\ell-1)} \cong
H_{m-\ell+1}(\Lambda^{\ell} (\Omega^1_{P_*|\Q} \otimes_{P_*} \Q))$$
because
$$ (SI)^{\circ (t)}(A) \cong Q_a((SI)^{\circ (t+1)}(A)) \cong
\Omega^1_{P_t|\Q} \otimes_{P_t} \Q$$
with $P_t = (SI)^{\circ (t+1)}(A)$.

For $n=2k-1$ however, we get symmetric powers of the K\"ahler
differentials and have
$$\mathrm{Tor}_{m+1-\ell}^\Gamma(\theta^\ell, \mathcal{L}(A;\Q)) \cong
(\mathbb{L}_mQ_{(2k-1)G}\bar{A})_{(2k-1)(\ell-1)} \cong
H_{m-\ell+1}(\mathrm{Sym}^\ell (\Omega^1_{P_*|\Q} \otimes_{P_*} \Q))$$
again with $P_t = (SI)^{\circ (t+1)}(A)$.
\end{proof}

\end{document}